\documentclass{article}

\usepackage{amsmath, amssymb, amsthm}
\usepackage[shortlabels]{enumitem}
\usepackage{subcaption}

\usepackage[a4paper, left=3.5cm,top=4cm,right=3.5cm,bottom=4cm]{geometry}

\usepackage[style=alphabetic]{biblatex}
\usepackage{tikz}
\usetikzlibrary{calc}

\tikzstyle{divisor}=[circle,very thick,draw,scale=0.4,fill=white]
\tikzstyle{vertex}=[rectangle,draw,scale=0.4,fill=black]

\usepackage{hyperref}

\author{Matthew Dupraz}
\title{Realizability of tropical canonical divisors and linear systems}

\newcommand{\R}{\mathbb{R}}
\newcommand{\N}{\mathbb{N}}
\newcommand{\Z}{\mathbb{Z}}
\newcommand{\Q}{\mathbb{Q}}
\newcommand{\T}{\mathbb{T}}

\newcommand{\fX}{\mathfrak{X}}
\newcommand{\fm}{\mathfrak{m}}
\newcommand{\fd}{\mathfrak{d}}

\newcommand{\cO}{\mathcal{O}}
\newcommand{\cK}{\mathcal{K}}

\newcommand{\proj}{\mathbb{P}}

\newcommand{\supp}{\operatorname{supp}}
\newcommand{\bend}{\operatorname{bend}}
\newcommand{\lcm}{\operatorname{lcm}}
\newcommand{\Div}{\operatorname{Div}}
\newcommand{\fdiv}{\operatorname{div}}
\newcommand{\val}{\operatorname{val}}
\newcommand{\ord}{\operatorname{ord}}
\newcommand{\dist}{\operatorname{dist}}
\newcommand{\Spec}{\operatorname{Spec}}
\newcommand{\rank}{\operatorname{rk}}
\newcommand{\coker}{\operatorname{coker}}
\newcommand{\Rat}{\operatorname{Rat}}
\newcommand{\PL}{\operatorname{PL}}
\newcommand{\Adj}{\operatorname{Adj}}
\newcommand{\diag}{\operatorname{diag}}
\newcommand{\im}{\operatorname{im}}
\newcommand{\relint}{\operatorname{relint}}
\newcommand{\Real}{\operatorname{Real}}
\newcommand{\trop}{\operatorname{trop}}

\newcommand{\floor}[1]{{\lfloor #1 \rfloor}}

\newcommand{\inner}[2]{{\langle #1, #2 \rangle}}

\newcommand{\boundary}[1]{{\partial #1}}
\newcommand{\interior}[1]{{#1^\circ}}
\newcommand{\closure}[1]{{\overline{#1}}}

\newcommand{\completion}[1]{{\widehat{#1}}}
\newcommand{\degout}{\deg^{\textrm{out}}}
\newcommand{\tropicalhodge}{\mathbb{P}\Omega\mathcal{M}_g^\textrm{trop}}
\newcommand{\realizablelocus}{\mathbb{P}\mathcal{R}}

\newtheorem{theorem}{Theorem}[section]
\newtheorem{corollary}{Corollary}[theorem]
\newtheorem{lemma}[theorem]{Lemma}
\newtheorem{proposition}[theorem]{Proposition}

\theoremstyle{definition}
\newtheorem{definition}[theorem]{Definition}
\newtheorem{remark}[theorem]{Remark}
\newtheorem{example}[theorem]{Example}
\newtheorem{question}[theorem]{Question}

\begin{document}

\begin{titlepage}
\newcommand{\HRule}{\rule{\linewidth}{0.5mm}} 

\center 
 

\vspace{3cm}
\textsc{\LARGE École polytechnique fédérale de Lausanne}\\[1.5cm] 
\textsc{\Large Master thesis}\\[0.5cm]
{\large MSc Mathematics}\\[1cm] 


\HRule \\[0.5cm] 
{ \huge \bfseries Tropical linear systems\\and the realizability problem}\\[0.5cm] 
\HRule \\[1.5cm]
 

{\Large Matthew \textsc{Dupraz}}\\[0.5cm]
{\large June 2024}\\[2cm]

{
\large
\begin{tabular}{ll}
    \textbf{Supervisors:} & Dr. Francesca \textsc{Carocci}, \emph{Université de
    Genève}\\
     & Dr. Dimitri \textsc{Wyss}, \emph{EPFL}
\end{tabular}
}

\vfill


\includegraphics[width=0.4\linewidth]{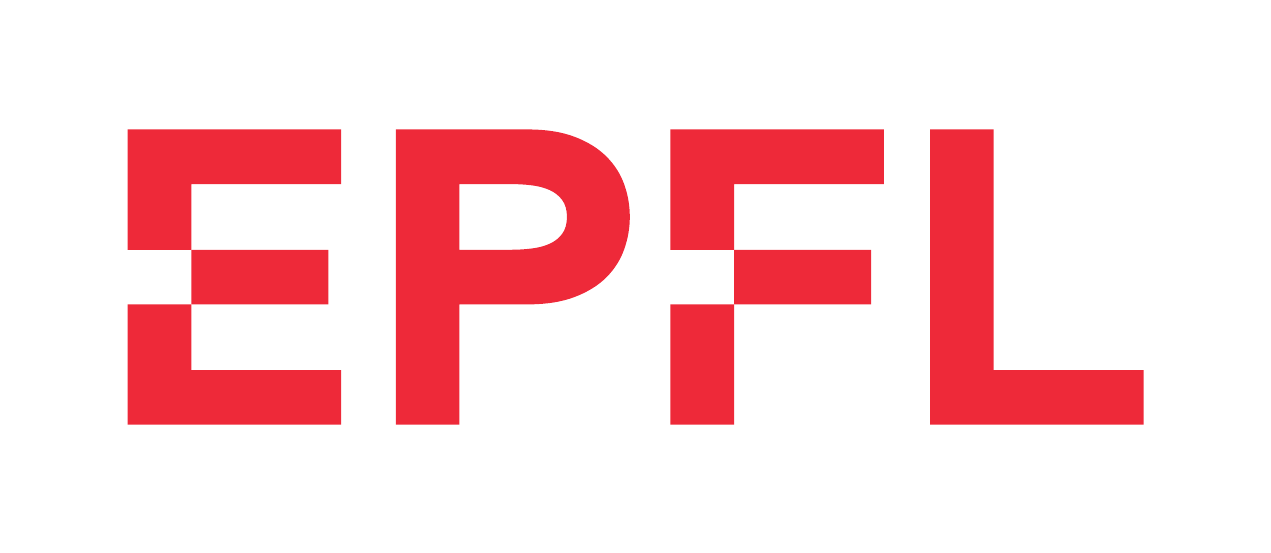}
 

\end{titlepage}

\begin{center}
    \textsc{\Large Tropical linear systems\\
        and the realizability problem}\\[0.5cm]
    Matthew \textsc{Dupraz}
\end{center}
\vspace{0.5cm}

\tableofcontents

\section{Introduction}

The goal of this thesis is to
explore linear systems on metric graphs, which despite being
relatively simple objects to understand, have much in common with their
counterparts on algebraic
curves.
There are suitable notions divisors, rational functions and linear
systems on metric graphs, which closely mimic how these objects behave on
algebraic curves. It turns out that this uncovers many interesting
connections between the world of algebraic geometry and combinatorics.
Baker and Norine defined in \cite{trop-rr} the notion of \emph{rank}
of a divisor, which behaves in very similar ways on algebraic curves and metric
graphs. For example, the Riemann-Roch theorem for algebraic curves may be stated
using the rank, and an amazing result shown in \cite{trop-rr} is that
the Riemann-Roch theorem holds also for metric graphs.

Complete linear systems on metric graphs 
have plenty of interesting combinatorial 
structure. On one hand, a complete linear system $|D|$
is an abstract polyhedral complex, and on the other hand,
the set of rational functions $R(D)$ associated to $|D|$ forms a
tropical module (a semi-module
equipped with the element-wise maximum and addition operations). 
The set $|D|$ may be identified with the tropical projectivization
$R(D)/\R$, and so one may
study the subspaces of $|D|$ that appear as the
projectivization of submodules of $R(D)$.
We will call such a subspace $\fd \subseteq |D|$ a \emph{tropical linear system}
(or a tropical linear \emph{series}).
In similar fashion to complete linear systems, it turns out that in 
many cases tropical linear systems also have an induced abstract
polyhedral complex structure.

The theories of linear systems on metric graphs and algebraic curves are far
from being just 
in simple analogy, as it is possible to link them via a process called
\emph{tropicalization}.
Given an algebraic curve with only ordinary double points as singularities,
one may associate to it a
graph, called its \emph{dual graph}.
When such a curve appears as a closed
fiber of a fibered surface, one may furthermore uniquely attribute edge lengths
to the dual graph and so give it the structure of a metric graph.
There is then a way to transfer divisors from the generic fiber of the surface
to the metric
graph through a process called \emph{specialization}. 
Matt Baker has shown in \cite{specialization-lemma} the
\emph{specialization lemma},
which states that the rank of a
divisor can only go up under specialization. This comparison theorem allows one
to derive results about divisors and linear systems on algebraic curves by
studying metric graphs.

One interpretation of the specialization lemma is that there are ``more"
divisors on metric graphs than on algebraic curves. It is then reasonable to ask
which divisors on metric graphs come from a divisor on an algebraic curve, if we
also require the rank to be preserved.
This question is called the realizability problem and there have been only a few
specific classes of divisors for which the realizable divisors were fully
characterized. For example in \cite{realizability-canonical} the authors give a
complete characterization of realizability for canonical divisors and this
result was later extended to pluri-canonical divisors in
\cite{realizability-pluricanonical}. In general, it is an important open problem
awaiting to be solved.

One object of study of this thesis is the set of
realizable divisors (also called the
\emph{realizability locus}) in the canonical linear system. We reinterpret 
the characterization for realizability from
\cite{realizability-canonical} and use the resulting criteria to show that the
realizability locus is an abstract
polyhedral complex and that it
is a tropical linear system.

A natural extension of the realizability question concerns 
the realizability of linear
systems. Tropicalizing a linear system on a curve
yields a tropical linear system and one may again ask the realizability
question in this context. This question is more complicated, because even if
all divisors
in a tropical linear system $\fd$ are realizable, it is possible that $\fd$
does not appear as the
tropicalization of a linear system on any given curve.
The theory of tropical linear systems is presented in
\cite{linsys-independence}, and further in \cite{kodaira-dimensions},
but it is a very new topic, and remains largely unexplored.

Since the rank of a linear system on an algebraic curve is equal to its
dimension as a projective space, it is natural to try to establish such a link
for metric graphs.
In this thesis we define a suitable notion of
local dimension of a tropical linear system and show 
that the local dimension is bounded from below by the rank. 
In \cite{linsys-independence}, the authors show that when the tropical linear
system is finitely generated and satisfies a further combinatorial condition,
the dimension may also be bound from above by the rank.
Along with the results from \cite{kodaira-dimensions}, this shows that
tropicalizations of linear systems on algebraic curves are
equi-dimensional abstract polyhedral complexes of dimension equal to the rank,
establishing a strong link
between the rank and dimension of realizable tropical linear systems and largely
limiting what kinds tropical linear systems may be realizable.

\subsubsection*{Structure of the thesis}

Section \ref{sec:metric-graphs} focuses on the theory of divisors and linear
systems on metric graphs. We start by covering the essential definitions
concerning metric graphs and divisors in subsections
\ref{ssec:metric-graphs} through \ref{sec:reduced-divisors}.
In subsection \ref{sec:tropical-modules} we will
introduce tropical modules and in
subsection \ref{sec:polyhedral-complex} we
describe how complete linear systems admit the
structure of an abstract polyhedral complex.
We then give a sufficient condition
for a subset to also admit the structure of an abstract polyhedral complex, and
describe how we can detect its dimension at a point.
In subsection \ref{sec:complete-linear-systems}
we show that the local dimension of a complete linear system is bounded from
below by the rank (Proposition \ref{prop:rank-dim}).
In subsection \ref{sec:structure-canonical}, we give
some characterizations of the canonical linear system. In particular we show
that the lower bound on the dimension is attained in the case of canonical
linear systems. In subsection 
\ref{sec:tropical-linear-systems} we extend Proposition \ref{prop:rank-dim} to
the setting of tropical linear systems (Corollary \ref{cor:complex-dim}).

In section \ref{sec:discrete-representations} we make the links between the
worlds of tropical and algebraic geometry. We first go into the details of the
tropicalization process in subsection \ref{sec:tropicalization}. In subsection
\ref{sec:specialization}, we
explain the specialization of divisors from algebraic curves to metric graphs
and the specialization lemma.
In subsection \ref{sec:realizability-canonical} we describe the condition for
realizability shown in \cite{realizability-canonical} and give a cleaner
characterization of inconvenient vertices. We then use this characterization in
subsection \ref{sec:canonical-linsys} to show that the realizability locus of
the canonical linear system is
tropically convex and an abstract polyhedral complex. In subsection
\ref{sec:cycles}
we give a sufficient condition for
realizability of canonical divisors (Proposition \ref{prop:no-disjoint-realizable}) and deduce that the realizable locus always
contains a maximal cell of dimension $g-1$. In subsection \ref{sec:lin-series} we explain that
specialization preserves linear equivalences and discuss the image of a linear
series under the specialization map. We then describe the advances made in
\cite{linsys-independence} and \cite{kodaira-dimensions} on this topic and
deduce that tropicalizations of linear series are equi-dimensional.

Finally, in section \ref{sec:discrete-representations}
we describe the theory of linear systems on graphs (without edge lengths) and
explain how it relates to the theory on metric graphs.
We describe useful results that can be used to work efficiently with
these discretizations and allow their implementation with algorithms.

\subsubsection*{Implementation}

I have used the concepts and results from Section
\ref{sec:discrete-representations}
to explore linear systems with a computer program.
Concretely, I wrote code for
working with metric graphs, which can among other things:
\begin{itemize}[itemsep=0em]
    \item Find all divisors in $|D|$ supported on a fixed model
    \item Find the extremals of $|D|$
    \item Check the realizability of a divisor in the canonical linear system
    \item Test whether a rational function belongs to the span of a generating
        set
    \item Find the maximal cells of $|D|$ and calculate their dimensions
\end{itemize}
This helped me build an intuition, find counter-examples and form hypotheses
regarding linear systems on metric graphs.

The code is freely accessible on following GitHub repository:
\url{https://github.com/MattDupraz/Graph-Linear-Systems.git}.

\subsubsection*{Acknowledgements}

I would like to express my deepest gratitude to my advisor, Francesca Carocci,
for her unwavering support and guidance, which extended well beyond the scope of
the thesis. I~have learned much from her not only
in terms of mathematics, but also on a personal level.

\section{Metric graphs and linear systems}
\label{sec:metric-graphs}

\subsection{Metric graphs}
\label{ssec:metric-graphs}

We will start by introducing the notion of \emph{metric graph}. Intuitively, a
metric graph is just a metric space isomorphic to the geometric realization
of a graph with given edge lengths. However, to be more precise, we will define
it using the notion of \emph{length space}. This approach
is largely inspired by \cite{metric-graph}.
For the definition of length space, we follow \cite{metric-geometry}.
Throughout this paper, we allow distance
functions that admit infinite values.

Let $(X, d)$ be a metric space. A \emph{path} is a continuous 
map $\gamma: [a, b] \to X$.
We will now define the length of a path. 
\begin{definition}\cite[Definition 2.3.1.]{metric-geometry}
    Let $\gamma: [a, b] \to X$ be a path. A \emph{partition} of $[a,b]$ is
    a finite collection of points $\{x_0, \dots, x_N\} \subseteq [a, b]$
    with
    \begin{equation*}
        a = x_0 < x_1 < \dots < x_N = b.
    \end{equation*}
    We define the \emph{length} of $\gamma$ as
    \begin{equation*}
        L(\gamma) = \sup \sum_{i = 1}^N d(\gamma(x_{i-1}), \gamma(y_i)),
    \end{equation*}
    where the supremum is taken over all the partitions of $[a, b]$.
    A curve is said to be \emph{rectifiable} if its length is finite.
\end{definition}

\begin{figure}[ht]
    \centering
    \begin{tikzpicture}
        \draw[red] (0,0) -- (2,2) -- (3,1) -- (5,0)
            -- (6,1) -- (5,2) -- (4.2,1.2);

        \draw[thick] (0,0) .. controls (0,1) and (1, 2) ..
              (2,2) .. controls (2.5,2) and (2.75,1.5) ..
              (3,1) .. controls (3.5,-0.2) and (4.8,0) ..
              (5,0) .. controls (5.5,0) and (6,0.5) ..
              (6,1) .. controls (6,1.5) and (5.5,2) ..
              (5,2) .. controls (4.5,2) and (4.2,1.5) .. (4.2, 1.2);
          \begin{scope}[every node/.style={vertex, red}]
            \node at (0,0) {};
            \node at (2,2) {};
            \node at (3,1) {};
            \node at (5,0) {};
            \node at (6,1) {};
            \node at (5,2) {};
            \node at (4.2,1.2) {};
        \end{scope}
    \end{tikzpicture}
    \caption{Example of a path and rectification}
    \label{fig:path}
\end{figure}
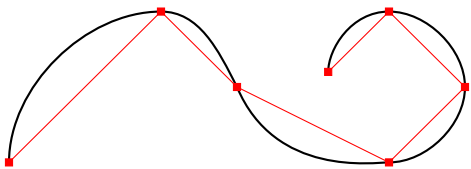

The notion of path length allows us to define a new distance on $X$.
\begin{definition}\cite[\S 2.1.2.]{metric-geometry}
    Let $(X, d)$ be a metric space. We define the
    \emph{induced intrinsic metric}
    to be 
    \begin{equation*}
        d_I(x, y) = \inf L(\gamma),
    \end{equation*}
    where the infimum is taken over all the paths $\gamma: [a, b] \to X$,
    with $\gamma(a) = x$ and $\gamma(b) = y$.
    If there is no path between $x$ and $y$ (when $X$ is disconnected), we let
    $d_I(x, y) = +\infty$.

    A metric space whose distance function is the same as the 
    induced intrinsic metric
    is called a \emph{length space}.
\end{definition}

\begin{remark}
    \label{rem:intrinsic-continuous}
    When $(X, d)$ is a metric space, the topology induced by the intrinsic
    metric $d_I$ is finer than the one induced by $X$. To see this, notice that
    for all $x, y \in X$,
    \begin{equation*}
        d_I(x, y) \geq d(x, y).
    \end{equation*}
    Indeed, when $x, y$ are connected by a path $\gamma$, then 
    $L(\gamma) \geq d(x, y)$ by definition of path length. When $x, y$ lie in
    different path-connected components of $X$, then $d_I(x, y) = +\infty$
    which also directly implies the above inequality.
    In other words, the identity map $(X, d_I) \to (X, d)$
    is continuous.
\end{remark}

\begin{definition}
    A \emph{metric graph} is a \emph{compact} length space
    $\Gamma$ such that each point $x
    \in \Gamma$ has a neighbourhood $U_x$ that is homeomorphic to
    $\bigsqcup_{i = 1}^v[0, \epsilon)/\sim$
    for some $\epsilon > 0$, where the
    equivalence relation $\sim$ identifies the zeroes of the intervals, and such
    that $x$ corresponds to the $0$ via this isomorphism. We call such a
    neighbourhood a \emph{star-shaped neighbourhood}.
    We say $v$ is the \emph{valence} of $x$ and denote it by $\val(x)$.
\end{definition}

\begin{remark}
    Some authors use the terminology \emph{abstract tropical curve} to designate
    metric graphs.
\end{remark}

\begin{remark}
    We necessarily have that the set of points $x \in \Gamma$ with
    $\val(x) \neq 2$ is finite.
\end{remark}

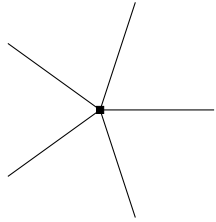
\begin{figure}[ht]
    \centering
    \begin{tikzpicture}[scale=1.5]
        \draw (0,0)--(0:1cm)
            (0,0)--(72:1cm)
            (0,0)--(72*2:1cm)
            (0,0)--(72*3:1cm)
            (0,0)--(72*4:1cm);
        \node[vertex] at (0,0) {};
    \end{tikzpicture}
    \caption{Star-shaped neighbourhood of a point of valence 5}
\end{figure}

\begin{definition}
    Let $V \subseteq \Gamma$ be a finite subset such that $\Gamma\setminus V$
    consists of disjoint union of open intervals.
    Then $V$ determines a \emph{model} 
    $G = (V, E)$ of the metric graph $\Gamma$, where $E$ is the set of 
    undirected edges
    corresponding to the open intervals of $\Gamma \setminus V$.
    For $e \in E$ an edge, we define $l(e)$ to be the length of the
    corresponding open interval.
\end{definition}

When $X$ is a length space and $\sim$ is an equivalence relation, we would like
to equip $X/\sim$ with the structure of a length space.
Following \cite[Definition 3.1.12.]{metric-geometry} we may define the following
\emph{semi}-metric on $X/\sim$:
\begin{equation*}
    d_\sim([x], [y]) = \inf\left\{\sum_{i = 1}^k d(p_i, q_i)\right\},
\end{equation*}
where the infimum is taken over sequences $p_1, \dots, p_k$ and $q_1, \dots,
q_k$ of points in $X$, such that $p_1 \sim x$, $q_k \sim y$
and $q_i \sim p_{i + 1}$ for $1 \leq i \leq k-1$.
This is a semi-metric, because it might happen that
$d_\sim([x], [y]) = 0$ even when $[x] \neq [y]$. A prototypical example of this
happening is the line with two origins obtaining by gluing two copies of $\R$
along $\R \setminus \{0\}$.

When $\Gamma$ is a metric graph, we would like to be able to glue some vertices
together to obtain a new graph. 
Let $u, v$ be two points of $\Gamma$, we may take the
quotient $\Gamma/\{u, v\}$, where the equivalence relation simply
identifies these two points. It turns out that the semi-metric $d_\sim$ defined
on this quotient is actually a metric.

\begin{proposition}
    The metric $d_\sim$ on $\Gamma/\{u, v\}$ is a well-defined metric.
\end{proposition}
\begin{proof}
    Clearly, $d_\sim$ is symmetric and non-negative and satisfies the triangle
    inequality. We have to verify that
    $d_\sim([x], [y]) = 0$ if and only if $[x] = [y]$.
    Suppose $d_\sim([x], [y]) = 0$.
    So for all $\epsilon > 0$ there exist sequences $(p_i)$, $(q_i)$ with the
    properties above, such that
    \begin{equation}
        \label{eq:inf-ineq}
        \inf\left\{\sum_{i = 1}^k d(p_i, q_i)\right\} < \epsilon.
    \end{equation}
    We may assume without loss of generality that $q_i \neq p_{i+1}$ for
    $1 \leq i \leq k-1$, since otherwise we have
    \begin{equation*}
        d(p_i, q_i) + d(p_{i+1}, q_{i+1}) \leq d(p_i, q_{i+1}) 
    \end{equation*}
    by the triangle inequality, so we could just remove the terms
    $q_i$ and $p_{i+1}$ from the sequences.
    So for $1 \leq i \leq k-1$, we may assume that
    $q_i, p_{i + 1} \in \{u, v\}$.

    If $x \notin \{u, v\}$, then $d(x, u), d(x, v) > \delta$ for some
    $\delta$ small enough, so (\ref{eq:inf-ineq}) implies that $k = 1$
    and $d(x, y) < \epsilon$ for all $\epsilon < \delta$, which in turn implies
    $d(x, y) = 0$ and so $x = y$. By symmetry we obtain the same result
    when $y \not\in \{u, v\}$,
    The last case is $x, y \in \{u, v\}$, but then $x \sim y$, so we are done.
\end{proof}

\begin{remark}
    The metric space $\Gamma/\{u, v\}$ is in fact a length space, as explained
    in \cite[\S3.1]{metric-geometry}, so it is in fact a metric graph,
    as around the image of $x \sim y$ we will again obtain a
    star-shaped neighbourhood of valence $\val(x) + \val(y)$.

    By induction, for any finite set of vertices $\{v_1, \dots, v_n\}$,
    the quotient space
    $\Gamma/\{v_1, \dots, v_n\}$ is also a metric graph. More in general,
    if $A_1, \dots, A_n$ are finite disjoint sets of vertices,
    we define $\Gamma/(A_1, \dots, A_n)$ to be the quotient by the equivalence
    relation $x \sim y$ if and only if
    $x = y$ or $\{x, y\} \subseteq A_i$ for some $i$.
    This is again a metric graph.

    As seen in
    \cite[Exercise 3.1.14.]{metric-geometry}, the topology induced by the metric
    coincides with the quotient topology, and so
    in particular the quotient map $\Gamma \to \Gamma/(A_1, \dots, A_n)$
    is continuous.
\end{remark}

\begin{definition}
    We say $\Gamma/(A_1, \dots, A_n)$ is a \emph{gluing} of $\Gamma$.
    Equivalently, we say that $\Gamma$ is a \emph{cut} of
    $\Gamma/(A_1, \dots, A_n)$.
\end{definition}

\begin{figure}[ht]
    \centering
    \begin{tikzpicture}
        \draw (0,0) -- (-1.4,0)
        (0,0) -- (1,1)
        (0,0) -- (1,-1);

        \begin{scope}[every node/.style=vertex]
            \node at (1,1) {};
            \node at (1,-1) {};
        \end{scope}

        \draw[line width=1.5pt,->] (1.75,0) -- (2.25,0);

        \begin{scope}[shift={(4.5,0)}, scale=1.4]
            \draw (0,0) -- (-1,0)
                (0,0) to[bend right=45] (1,0)
                (0,0) to[bend left=45] (1,-0);

            \node[vertex] at (1,0) {};
        \end{scope}
    \end{tikzpicture} 
    \caption{Example of a gluing of two vertices}
\end{figure}
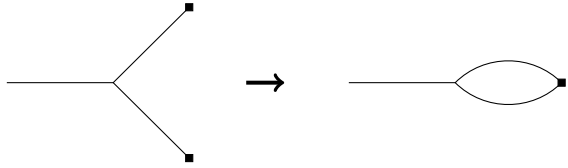

\begin{definition}
Suppose $G = (V, E)$ is a graph and $l: E \to \R_{>0}$ a map that assigns to
each edge a length. We may construct from this a metric graph. Let
\begin{equation*}
    \mathcal{E} = \bigsqcup_{e \in E}[0, l(e)].
\end{equation*}
The metric on the disjoint union is given by
\begin{equation*}
    d((x_1, e_1), (x_2, e_2)) = 
    \begin{cases}
        |x_1 - x_2| &\textrm{if } e_1 = e_2,\\
        0 &\textrm{otherwise.}
    \end{cases}
\end{equation*}
This clearly gives $\mathcal{E}$ the structure of length space and of a
metric graph.
Fix an ordering $v_1, \dots, v_n$ on vertices and suppose $e$ is an edge between
$v_i, v_j$ with $i \leq j$. We denote $s(e) = v_i$ and $t(e) = v_j$.
For each $i \in \{1, \dots, n\}$ let
\begin{equation*}
    V_i = \{(e, 0) \in \mathcal{E}: s(e) = v_i\}
    \cup \{(e, l(e)) \in \mathcal{E}, t(e) = v_i\}
\end{equation*}
Let $\Gamma = \mathcal{E}/(V_1, \dots, V_n)$. We may identify
$V = \{v_1, \dots, v_n\}$ with the images of $V_1, \dots, V_n$ under the
gluing and this induces a model on $\Gamma$ that agrees with $G$.
In other words, we have constructed a metric
graph that admits $G$ as a model and whose edge lengths agree with the
function $l$. We say $\Gamma$ is a \emph{realization}
of $(G, l)$.
\end{definition}

\begin{definition}
    We define the \emph{genus} $g(\Gamma)$ of a metric graph $\Gamma$
    to be its first Betti number. In other
    words it corresponds to the maximal number of
    independent cycles it contains.
\end{definition}

\begin{remark}
    \label{rem:genus-formula}
    If a metric graph $\Gamma$ with $n$ connected components
    admits a model $G = (V, E)$,
    then we have the relation 
    \[g(\Gamma) = |E| - |V| + n.\]
    Indeed, the simplicial 
    homology groups of $\Gamma$ are calculated from the chain complex
    \begin{align*}
        \cdots \to 0 \to \Z^E &\xrightarrow{\partial_1} \Z^V \to 0\\
        (u, v) &\mapsto u - v,
    \end{align*}
    where we fixed an arbitrary orientation for each edge.
    
    We have that $H^0(\Gamma) = \coker(\partial_1)$. Vertices that are joined by
    an edge are identified in the cokernel, so we deduce that
    $\rank H^0(\Gamma) = n$. We have the exact sequence
    \begin{equation*}
        0 \to \ker \partial_1 \to \Z^E \to \Z^V \to \coker \partial_1 \to 0,
    \end{equation*}
    from where it follows that
    \begin{equation*}
        \rank \ker \partial_1 - |E| + |V| - \rank \coker(\partial_1) = 0.
    \end{equation*}
    The formula for the genus then follows by remarking that
    $H^1(\Gamma) = \ker(\partial_1)$.
\end{remark}

We will now define tangent vectors on metric graphs in analogy to the definition
of tangent vectors on manifolds via tangent curves.
\begin{definition}
    For $\epsilon > 0$,
    let $I_{\epsilon, x}(\Gamma)$ be the set of isometries
    $\gamma: [0, \epsilon) \to \Gamma$, with $\gamma(0) = x$.
    For $\epsilon > \epsilon'$, we have a natural map
    $I_{\epsilon, x}(\Gamma) \to I_{\epsilon', x}(\Gamma)$
    given by the restriction, and this defines a direct system over
    $[0, \infty)$.
    Let $T_x\Gamma = \varinjlim I_{\epsilon, x}$ be the direct limit of this
    system.
    We call this the set of (unit) \emph{tangent vectors} of $\Gamma$ at $x$.
\end{definition}
\begin{remark}
    For $\epsilon$ small enough, the open ball
    $B(x, \epsilon)$ is a star-shaped neighbourhood, so in this case
    the elements of
    $I_{\epsilon, x}(\Gamma)$ correspond 
    to the identification of $[0, \epsilon)$
    to one of the copies of $[0, \epsilon)$ in
    \begin{equation*}
        B(x, \epsilon) \cong \bigsqcup_{i = 1}^{\val(x)}[0, \epsilon)/\sim.
    \end{equation*}
    In other words, there is a bijective correspondence 
    between the tangent vectors at $x$ and the half-edges of $\Gamma$
    adjacent to $x$.
\end{remark}

\begin{definition}
    \label{def:completion}
    Let $U \subseteq \Gamma$
    be an open subset with a finite number of connected components.
    We endow $U$ with the induced intrinsic metric
    and let $\hat{U}$ be the completion of $U$ with respect to this metric.
    Another way to see $\hat{U}$ is as the space obtained from $U$ by adding
    a point to each open half-edge of $U$. From this description it is clear
    that $\hat{U}$ is also a metric graph.
    
    By remark \ref{rem:intrinsic-continuous}, the inclusion
    $U \hookrightarrow \Gamma$ is continuous and so as $\Gamma$ is compact,
    there is unique map $\phi: \hat{U} \to \Gamma$ that extends 
    $U \hookrightarrow \Gamma$. Its image is the closure of $U$ in $\Gamma$.

    For any model of $\Gamma$,
    the metric graph $\hat{U}$ naturally inherits a model structure,
    which is the minimal
    model for the property that it contains
    $V\cap U \hookrightarrow \hat{U}$ in its set of vertices.
\end{definition}

\begin{figure}[ht]
    \centering
    \begin{tikzpicture}[scale=1.5]
        \begin{scope}[shift={(-1.5,0)}]
            \coordinate (a) at (0,0);
            \coordinate (b) at (0,1);
            \coordinate (c) at (-30:1cm);
            \coordinate (d) at (210:1cm);

            \draw (d)--(b);
            \draw[red] (a)--(b)--(c)--(d)--(a)--(c);
            \begin{scope}[every node/.style=vertex]
                \node[black] at (a) {};
                \node[red] at (b) {};
                \node[red] at (c) {};
                \node[red] at (d) {};
            \end{scope}
        \end{scope}

        \draw[line width=2pt,->] (-0.2,0) -- (0.2,0);
         
        \begin{scope}[shift={(1.5,0)}]
            \coordinate (b) at (0,1);
            \coordinate (c) at (-30:1cm);
            \coordinate (d) at (210:1cm);
            \coordinate (b2) at (0,0.25);
            \coordinate (c2) at (-30:0.25cm);
            \coordinate (d2) at (210:0.25cm);

            \draw[red] (b2)--(b)--(c)--(c2)
                (c)--(d)--(d2);
            \begin{scope}[every node/.style={vertex, red}]
                \node at (b) {};
                \node at (b2) {};
                \node at (c) {};
                \node at (c2) {};
                \node at (d) {};
                \node at (d2) {};
            \end{scope}
        \end{scope}
    \end{tikzpicture}
    \caption{Example of completion of an open subgraph (in red)}
\end{figure}
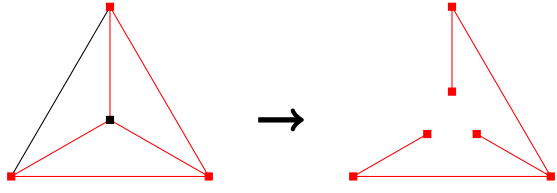

\begin{definition}
    If $U$ is an open subset with a finite number of connected components, we
    define its genus $g(U)$ to be the genus of its completion $\hat{U}$.
\end{definition}

We will now prove a useful lemma which may be used to calculate the genus of the
graph obtained after cutting the graph in a finite number of points.

\begin{lemma}
    \label{lemma:glueing-genus}
    Suppose $\Gamma$ is connected.
    Let $A \subset \Gamma$ be a finite set of points.
    Then
    \begin{equation*}
        g(\Gamma) = g(\Gamma \setminus A) + \sum_{x \in A}(\val(x) - 1) + 1 - N,
    \end{equation*}
    where $N$ is the number of connected components of $\Gamma \setminus A$
\end{lemma}
\begin{proof}
    We will denote $V(G)$ and $E(G)$ the set of vertices and edges
    of a graph $G$ respectively.
    Up to subdividing the model of $\Gamma$, we may assume the set $A$ is
    contained in $V(\Gamma)$. 

    For $C$ a connected component of $\Gamma \setminus A$, $\hat{C} \to \Gamma$
    is a one-to-one mapping, except for points laying above some $x \in A$.
    For such $x$, there are exactly $\val_{\closure{C}}(x)$ points in the
    preimage.

    We deduce that
    \begin{equation*}
        \#V(\hat{C}) = \#(V(\Gamma) \cap C) + \sum_{x \in A}
        \val_\closure{C}(x).
    \end{equation*}
    By summing over the connected components, we get that
    \begin{align*}
        \sum_C\#V(\hat{C}) &= \#(V(\Gamma) \setminus A) + \sum_{x \in A}
        \val(x)\\
        &= \#V(\Gamma) + \sum_{x\in A}(\val(x) - 1)
    \end{align*}

    We get by Remark \ref{rem:genus-formula} that
    \begin{align*}
        \sum_{C}g(\hat{C}) &= \sum_{C}(\#E(\hat{C}) -
        \#V(\hat{C}) + 1)\\
        &= \#E(\Gamma) - \#V(\Gamma)
        - \sum_{x\in A}(\val(x) - 1) + N\\
        &= g(\Gamma) - 1 - \sum_{x\in A}(\val(x) - 1) + N
    \end{align*}
    whence the result follows directly from the fact that
    $g(C \setminus A) = \sum_C g(\hat{C})$.
\end{proof}

\begin{definition}
    Let $v \in \Gamma$ a vertex,
    then for $U = \Gamma \setminus \{v\}$,
    we have that as a set
    $\completion{U} = U \sqcup \{v_1, \dots, v_n\}$,
    where $n = \val(v)$.
    The set $A = \{v_1, \dots, v_n\}$ corresponds naturally to the
    sets of tangent vectors $T_{v}\Gamma$.
    Let $\zeta \in T_v\Gamma$ be a tangent and let
    and let $S \subseteq A$, be the subset of points corresponding to the
    tangents other than $\zeta$.
    Then we have that
    the quotient map $\completion{U} \to \Gamma$ factors
    as 
    \[\completion{U} \to \completion{U}/S \to \Gamma.\]
    We say $\completion{U}/S$ is the \emph{cut of $\Gamma$ obtained by
    cutting along $\zeta$}.
\end{definition}

\subsection{Divisors and complete linear systems}
\label{sec:linear-systems}

\begin{definition}
    A \emph{divisor} on $\Gamma$ is an element of the free abelian group
    generated by the points of $\Gamma$, which we denote by $\Div(\Gamma)$.
    An element of this group is written as
    \begin{equation*}
        D = \sum_{x \in \Gamma} D(x) \cdot x,
    \end{equation*}
    where $D(x) = 0$ for all but finitely many $x$. Examples of two divisors are
    depicted in Fig. \ref{fig:divisor-ex}.
    
    We define the \emph{support} of $D$ to be
    \[\supp(D) := \{x \in \Gamma: D(x) \neq 0\}.\]
    We say $D$ is \emph{effective}, denoted by $D \geq 0$,
    when $D(x) \geq 0$ for all $x \in \Gamma$.
    The \emph{degree} of $D$ is the sum of its coefficients, that is
    \[\deg(D) = \sum_{x \in \Gamma}D(x).\]
    When $Z$ is a subgraph of $\Gamma$, we call
    \[D|_Z := \sum_{x \in Z} D(x) \cdot x\]
    the \emph{restriction} of $D$ to $Z$.
\end{definition}

\begin{definition}
    The \emph{canonical divisor} of $\Gamma$ is the divisor defined by
    \begin{equation*}
        K = \sum_{x \in \Gamma} (\val(x) - 2)\cdot x
    \end{equation*}
\end{definition}

\begin{remark}
    The canonical divisor $K \in \Div(\Gamma)$ has $\deg K = 2g - 2$.
\end{remark}

\begin{figure}[ht]
    \begin{subfigure}[h]{0.4\linewidth}
        \centering
        \begin{tikzpicture}[scale=4]
            \draw (0,0) to[bend right=50]
                node[pos=3/5, divisor, label=1] {}
                node[pos=2/5, divisor, label=1] {}
                (1,0);
            \draw (0,0) to[bend left=50] (1,0);
            \draw (0,0) -- (1,0);
            \node[vertex] at (0, 0) {};
            \node[vertex] at (1, 0) {};
        \end{tikzpicture}
        \caption{Divisor of degree $2$}
    \end{subfigure}
    \hfill
    \begin{subfigure}[h]{0.4\linewidth}
        \centering
        \begin{tikzpicture}[scale=4]
            \draw (0,0) to[bend right=50] (1,0);
            \draw (0,0) to[bend left=50] (1,0);
            \draw (0,0) -- (1,0);
            \node[divisor, label=1] at (0, 0) {};
            \node[divisor, label=1] at (1, 0) {};
        \end{tikzpicture}
        \caption{Canonical divisor}
    \end{subfigure}

    \caption{Two divisors on the same metric graph. The points in the support of
    the divisors are represented using circles and labeled with their
    multiplicity.}
    \label{fig:divisor-ex}
\end{figure}
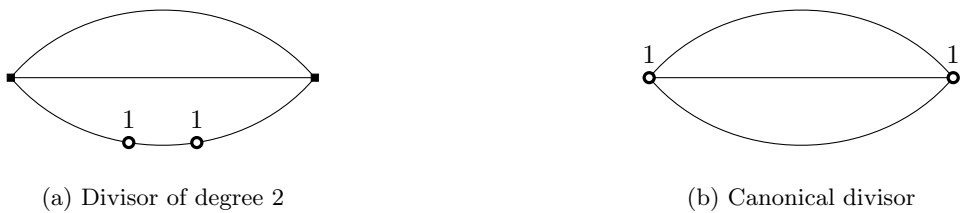

\begin{definition}
    A piece-wise linear (PL) function is a continuous function
    $f: \Gamma \to \R$ for which there exists a model $G = (V, E)$ such that
    $f$ is linear when restricted to the edges $e \in E$.
    We denote the set of piece-wise linear functions on $\Gamma$
    by $\PL(\Gamma)$.
    A \emph{rational} function is a PL function with integral slopes,
    and we denote the set of rational functions on $\Gamma$ by 
    $\Rat(\Gamma)$.

    Since $\Gamma$ is compact, the image of a PL function is compact, and so
    this lets us define a norm on $\PL(\Gamma)$ by
    \begin{equation*}
        \|f\|_\infty = \max f - \min f.
    \end{equation*}

    Let $f$ a rational function $f$ on $\Gamma$.
    For any $\zeta \in T_x\Gamma$, represented by an isometric path
    $\gamma: [0, \epsilon) \to \Gamma$,
    we define the slope of $f$ along $\zeta$ by
    \begin{equation*}
        s_\zeta(f) := \lim_{t \to 0}\frac{f(\gamma(t)) - f(x)}{t}.
    \end{equation*} 
    Since $f$ is piece-wise linear, $f\circ \gamma|_{[0, \delta]}$
    is linear for some
    $\delta > 0$ small enough,
    so this is well-defined, and clearly this
    does not depend on the choice of $\gamma$.
    
    The \emph{order} of $f$ at $x$, denoted by $\ord_x(f)$ is the sum of the
    outgoing slopes of $f$ along each edge emanating from $x$.
    In other words,
    \begin{equation*}
        \ord_x(f) = \sum_{\zeta \in T_x\Gamma} s_\zeta(f).
    \end{equation*}
    The \emph{principal divisor}
    associated to $f$ is the divisor given by
    \begin{equation*}
        \fdiv(f) := \sum_{x \in \Gamma} \ord_x(f)\cdot x.
    \end{equation*}
    Note that this is well defined, as for any model $(V, E)$ such that $f$ is
    linear when restricted to the edges,
    we have that $\ord_x(f) = 0$ for all $x \in \Gamma
    \setminus V$ and $V$ is finite.

    We define the \emph{bend locus} of $f$, denoted by $\bend(f)$,
    to be the support of the associated divisor
    $\fdiv(f)$.
\end{definition}

\begin{proposition}
    \label{prop:rational-degree}
    For any rational function $f$ on a compact metric graph $\Gamma$,
    we have that $\deg(\fdiv(f)) = 0$.
\end{proposition}

\begin{proof}
    Let $G = (V, E)$ a model of $\Gamma$ containing $\bend(f)$ in its set of
    vertices. For $e$ an edge and $x$ one of its vertices, there is a unique
    tangent vector $\zeta \in T_x\Gamma$ that comes from some
    $\gamma: [0, \epsilon) \to \Gamma$ whose image lies inside the closure
    of $e$ (in the future 
    we will simply say that $\zeta$ is the tangent of $x$ along $e$,
    and $s_\zeta(f)$ is the outgoing slope of $f$ at $x$ along $e$).
    Denote $\zeta_{e, 1}, \zeta_{e, 2}$ the two tangent vectors
    corresponding to the vertices of $e$. Since $f$ is linear along each edge,
    it follows that
    $s_{\zeta_{e, 1}}(f) = -s_{\zeta_{e, 2}}(f)$.
    We the obtain the desired result, as
    \begin{align*}
        \deg(\fdiv(f)) &= \sum_{x \in V}\sum_{\zeta \in T_x\Gamma} s_\zeta(f)\\
        &= \sum_{e \in E}(s_{\zeta_{e,1}}(f) + s_{\zeta_{e, 2}}(f)) = 0\\
    \end{align*}
\end{proof}

\begin{definition}
    A divisor $D \in \Div(\Gamma)$ is called \emph{principal} when there exists
    a rational function $f$ with $D = \fdiv(f)$.

    Two divisors $D, D' \in \Div(\Gamma)$ are said to be linearly equivalent, 
    denoted by $D \sim D'$,
    when $D - D'$ is principal. 
\end{definition}

\begin{remark}
    By Proposition \ref{prop:rational-degree}, it follows that that
    any two linearly equivalent divisors have the same degree.
\end{remark}

\begin{remark}
    The two divisors depicted in Figure \ref{fig:divisor-ex} are linearly
    equivalent.
\end{remark}

\begin{definition}
    Let $D$ be an effective divisor,
    we define $R(D)$ to be the set of rational functions
    $f$ such that $D + \fdiv(f) \geq 0$.
\end{definition}

\begin{definition}
    For $D$ an effective divisor, we define the
    \emph{complete linear system}
    associated to $D$ to
    be the set
    \begin{equation*}
        |D| := \{D' \geq 0: D' \sim D\}.
    \end{equation*}
\end{definition}

\begin{remark}
    Let $R(D)/\R$ be the quotient of $R(D)$ modulo tropical scaling, that is,
    we take the quotient by the equivalence relation defined 
    by $f \sim g$ if and only if
    $f = c + g$ for some $c \in \R$. Then we have a bijective
    correspondence
    \begin{align*}
        R(D)/\R &\to |D|\\
        f &\mapsto D + \fdiv(f)
    \end{align*}
\end{remark}

The set $\Div_d^+(\Gamma)$ 
of divisors of degree $d$ on $\Gamma$ may be naturally identified with
the symmetric product $\Gamma^d/S_d$. The latter is a topological space, and so
we may give $\Div_d^+(\Gamma)$ the structure of a topological space induced by
this identification. Since $|D|$ is a subset of $\Div_d^+(\Gamma)$, we may equip
it with the subspace topology. As seen below, $|D|$ naturally admits the
structure of a metric space.

\begin{proposition}
    The norm $\|\cdot\|_\infty$ on $PL(\Gamma)$
    induces a metric on $|D|$, which we will also denote 
    by $d_\infty$.
    Furthermore, the topology induced by this metric agrees with that
    induced by the inclusion $|D| \hookrightarrow \Gamma^d/S_d$.
\end{proposition}

\begin{proof}
    For any divisors $D_1, D_2 \in |D|$, there is some $f \in \Rat(\Gamma)$ such
    that $D_1 = D_2 + \fdiv(f)$. We define the metric $d_\infty$ by
    \begin{equation*}
        d_\infty(D_1 , D_2) = \|f\|_\infty.
    \end{equation*}
    We verify that this is well-defined. Suppose $\fdiv(f) = \fdiv(g)$. This
    implies in particular that $\fdiv(f - g) = 0$ and so $f - g$ is constant.
    Indeed, if $f - g$ was not constant, then if $Z$ was the subgraph of
    $\Gamma$ on which $f-g$ admits its minimum, there would be some $x \in
    \boundary{Z}$, and for such an $x$ we would necessarily have
    $\fdiv(f - g)(x) > 0$. Now, it follows directly from the definition of
    $\|\cdot\|_\infty$ that for any $c \in \R$,
    $\|f + c\|_\infty = \|f\|_\infty$, and hence taking $c = g - f$, we get that
    $\|f\|_\infty = \|g\|_\infty$.

    The fact that the topology induced by this metric agrees with that
    induced from $\Gamma^d/S_d$ is detailed in \cite[Proposition
    B.1]{luo-idempotent}.
\end{proof}

\begin{definition}
    \label{def:rank}
    The \emph{rank} of a divisor $D \in \Div(\Gamma)$ is the number
    \begin{equation*}
        r(D) := 
            \max \{d \in \N \mid |D - E| \neq \emptyset\textrm{
            for all effective divisor $E$ of degree $d$}\},
    \end{equation*}
    where if $|D| = \emptyset$ we set $r(D) = -1$.
\end{definition}

\begin{theorem}[Tropical Riemann-Roch] \cite[Theorem 1.12]{trop-rr}
    \label{thm:trop-rr}
    Let $D \in \Div(\Gamma)$ be a divisor, then
    \begin{equation*}
        r(D) - r(K-D) = \deg(D) - g + 1
    \end{equation*}
\end{theorem}

\begin{definition}
    For a closed, not necessarily connected
    subgraph $Z \subseteq \Gamma$, we define
    \begin{equation*}
        Z_\epsilon := \{x \in \Gamma: \dist(x, Z) < \epsilon \}.
    \end{equation*}
    A \emph{chip firing move} is the data of a closed subgraph
    $Z \subseteq \Gamma$ with a finite number of connected components,
    and a distance $\epsilon > 0$ such that 
    $Z_\epsilon \setminus Z$ is a disjoint union of open intervals.

    To such data we can associate the rational function
    \begin{equation*}
        CF(Z, \epsilon)(x) := -\min\{\dist(x, Z), \epsilon\}.
    \end{equation*} 
    This function is identically $0$ on $Z$, it is identically $-\epsilon$
    on $Z_\epsilon^c$ and it interpolates linearly between these two
    regions on 
    $Z_\epsilon \setminus Z$.

    For $D \in \Div(\Gamma)$ a divisor, we say that we obtain a divisor $D'$
    by \emph{firing} $Z$ (by a distance $\epsilon$) when
    $D' = D + \fdiv(CF(Z, \epsilon))$.
    We say that $Z \subseteq \Gamma$
    \emph{can fire} if there is some $\epsilon > 0$,
    such that for all $x \in \boundary{Z}$,
    $D(x) + \ord_xCF(Z, \epsilon) \geq 0$. 
\end{definition}

\begin{remark}
    When $Z \subset \Gamma$ is a closed subgraph and $x \in \boundary{Z}$,
    we define $\degout_Z(x)$ to be the valence of $x$ in the closed
    subgraph $\Gamma \setminus \interior{Z}$. It follows that for $x \in
    \boundary{Z}$,
    $\ord_xCF(Z, \epsilon) = -\degout_Z(x)$ and so
    $Z \subseteq \Gamma$ 
    can fire if and only if
    for all $x \in \boundary Z$, $D(x) \geq \degout_Z(x)$.
\end{remark}

\begin{remark}
    It follows directly from the definition that $\|CF(Z, \epsilon)\|_\infty =
    \epsilon$.
\end{remark}

\begin{definition}
    A \emph{weighted chip firing move} is a non-constant
    rational function $f$, for which
    there exist two disjoint
    proper closed subgraphs $Z_1$ and $Z_2$, such that
    $\Gamma \setminus (Z_1 \cup Z_2)$ consists only of open segments
    such that $f$ is constant on $Z_1$
    and $Z_2$ and linear on each component of $\Gamma \setminus (Z_1 \cup Z_2)$.
\end{definition}

\begin{lemma}
    \label{lemma:weighted-cf-decomp}
    Every weighted chip firing move $f$ can be written as a sum of chip firing
    moves (up to a constant)
    \begin{equation*}
        f = f_1 + \dots + f_n,
    \end{equation*}
    where $\|f_i\|_\infty \leq \|f\|_\infty$. Furthermore, if $f \in R(D)$ for
    some effective divisor $D$, then $f_{k} \in
    R(D + \fdiv(f_1 + \dots + f_{k-1}))$
    for all $k$,
    that is $f_1, \dots, f_n$ is a sequence of legal chip firing moves.
\end{lemma}

\begin{proof}
    We will proceed as in \cite[Lemma 1]{linsys}.
    Let $Z_1, Z_2$ be as in the definition of weighted chip firing move and
    let $d = f(Z_2) - f(Z_1)$. Without loss of generality,
    suppose that $d > 0$.
    Denote $L_1, \dots, L_r$ the
    open segments making up $\Gamma \setminus (Z_1 \cup Z_2)$.
    Let $l_i$ be the length of $L_i$.
    Let also $s_i > 0$ be the slope of $f$ along $L_i$ from $Z_1$ to $Z_2$ (so
    that $s_i = d/l_i$),
    and let $s = \lcm(s_1, \dots, s_r)$.

    Let $k_i = s/s_i$ and $\delta = d/s$. For $j = 0, \dots, s - 1$, we let
    $Y_j$ be the subgraphs obtained by attaching the adjacent subsegment of
    $L_i$ of length $\floor{j/s_i}\delta$ to $Z_2$.
    We then define $f_j := CF(Y_j, \delta)$. Let $g = f_1 + \dots + f_{s-1}$.
    Clearly $g$ is constant on $Z_1$ and $Z_2$ and its slope along any given
    $L_i$ is by definition equal to $s_i$, hence $f - g$ is a constant.

    Suppose $f \in R(D)$ for some effective divisor $D$.
    To show that $f_1, \dots, f_{s-1}$ is a sequence of legal chip firing moves,
    we will focus on a single $L_i$. Identify $L_i = (0, l_i)$, where we orient
    $L_i$ from $Z_2$ to $Z_1$. 
    We have that $g_k := f_1 + \dots + f_k$ is always concave on $L_i$.
    Indeed, let $\zeta_x$ be the tangent vector at $x \in L_i$,
    then by the definition of the $f_j$, we have that
    \begin{equation*}
        s_{\zeta_x}(f_1 + \dots + f_k) =
        \begin{cases}
            -s_i &\textrm{if } x\in (0, \floor{k/s_i}\delta),\\
            -s_i\{k/s_i\} & \textrm{if } x \in 
            [\floor{j/s_i}\delta, (\floor{k/s_i} + 1)\delta,\\
            0 &\textrm{otherwise.}
        \end{cases}
    \end{equation*}
    Here $\{a\} := a - \floor{a}$ denotes the fractional part of $a$.
    So $g_k$ is concave on $L_i$, whence
    \[\ord_x g_k \geq 0 \geq -D(x)\]
    for any $x \in L_i$.
    If $x$ is the point corresponding to $0 \in \closure{L_i}$,
    then the slope of $g_k$ at $x$ along $L_i$
    is at least $-s_i$, so we deduce that $\ord_x g_k \geq \ord_x f \geq -D(x)$.
    Finally, if $x$ is the point corresponding to $l_i \in \closure{L_i}$, then
    the slope of $g_k$ at $x$ along $L_i$ (in the opposite direction) is
    non-negative, and hence $\ord_x g_k \geq 0 \geq -D(x)$.
    We deduce that $g_k \in R(D)$, or equivalently that
    \[f_k \in R(D + \deg(f_1 + \dots + f_{k-1})).\]
\end{proof}

\begin{lemma}
    \label{lemma:tropical-sum}
    Every tropical rational function is a sum of chip firing moves (up to a
    constant). If we denote the sum by 
    $f = f_1 + \dots + f_n$, where the $f_i$
    are the chip firing moves, then the $f_i$ can be chosen such that
    $\|f_i\|_\infty \leq \|f\|_\infty$ and furthermore
    if $f \in R(D)$ for some divisor $D$,
    then $f_k \in R(D + \fdiv(f_1 + \dots + f_{k-1}))$ for all $k$, that is
    $f_1, \dots, f_n$ is a sequence of legal chip firing moves.
\end{lemma}

\begin{proof}
    We will proceed as in \cite[Lemma 2]{linsys}.
    Let $Y = f(\bend(f) \cup V)$, where $V$ is a set of vertices for any chosen
    model of $\Gamma$. Then $Y$ is finite so denote $y_1 > \dots > y_r$ its
    elements.  By
    construction we have that
    \begin{equation*}
        g_i := \max\{\min\{f, y_i\}, y_{i + 1}\}   
    \end{equation*}
    is a weighted chip-firing move.
    Note that
    \begin{equation*}
        \|g_i\|_\infty = y_{i+1} - y_i \leq y_r - y_1 = \|f\|_\infty,
    \end{equation*}
    and that for all $k$, 
    \[g_1 + \dots + g_k = \max\{f, y_{k+1}\} + c_k,\]
    where $c_k$
    is some constant.
    We show this by induction.
    We have that $g_1 = \max\{f, y_2\}$, so we can set $c_1 = 0$ and
    \begin{equation*}
        \max\{f, y_{k}\} + c_{k-1} + g_k =
        \max\{f, y_k\} + \max\{\min\{f, y_k\}, y_{k + 1}\} + c_{k - 1}.
    \end{equation*}
    Evaluated in $x \in \Gamma$, we obtain
    \begin{equation*}
        \begin{cases}
            f(x) + y_k + c_{k-1}&\textrm{ if }f(x) \geq y_k,\\
            y_k + f(x) + c_{k-1}&\textrm{ if }y_{k+1} \leq f(x) \leq y_{k},\\
            y_k + y_{k+1}&\textrm{ if }f(x) \leq y_{k+1}.
        \end{cases}
    \end{equation*}
    And hence if we set $c_{k} = c_{k-1} + y_k$, we get that
    \begin{equation*}
        \max\{f, y_{k}\} + c_{k-1} + g_k = \max\{f, y_{k+1}\} + c_k.
    \end{equation*}

    By Lemma \ref{lemma:weighted-cf-decomp}, we have
    that $g_k = f^{(k)}_1 + \dots + f^{(k)}_{n_k}$ with
    $\|f_i^{(k)}\|_\infty \leq \|g_k\|_{\infty}$.
    So let $D$ an effective divisor such that
    $f \in R(D)$, then since
    $g_1 + \dots + g_k = \max\{f,  y_{k+1}\} + c_k$,
    we have that $g_1 + \dots + g_k \in R(D)$.
    This implies that $g_k \in R(D + \fdiv(g_1 + \dots + g_{k-1}))$
    and hence by Lemma \ref{lemma:weighted-cf-decomp} we have that
    $$f^{(k)}_l \in R(D + \fdiv(g_1 + \dots +
    g_{k-1} + f^{(k)}_1 + \dots + f^{(k)}_{l-1}))$$
    for all $l$.
    
    Now it is clear that the desired properties hold if we set
    \begin{equation*}
        (f_1, \dots, f_n) = (f^{(1)}_1, \dots, f^{(1)}_{n_1}, \dots, 
        f^{(r-1)}_1, \dots, f^{(r-1)}_{n_{r-1}}).
    \end{equation*}
\end{proof}

\subsection{Reduced divisors}
\label{sec:reduced-divisors}

We now describe reduced divisors, which are distinguished divisors in a linear
system. The purpose of this subsection is to introduce some theoretical
background that will be useful in section \ref{sec:discrete-representations} and
is used in the implementation.

\begin{definition}
    Let $v \in \Gamma$ be a point. We say a divisor $D \in \Div(\Gamma)$
    is \emph{effective away from $v$} if $D(x) \geq 0$ for all $x \neq v$.

    A \emph{$v$-reduced} divisor is a divisor $D \in \Div(\Gamma)$ that is
    effective away from $v$ and such that for all subgraphs
    $Z \subset V$ with $v \notin Z$, $Z$ cannot fire with respect to $D$.
\end{definition}

\begin{proposition}
    \label{prop:reduced-existence-unicity}
    \cite[Theorem 2]{amini-reduced}
    Let $D \in \Div(\Gamma)$. There exists a unique $v$-reduced divisor linearly
    equivalent to $D$.
\end{proposition}

\begin{definition}
    Let $f$ a piece-wise linear function on $\Gamma$.
    We say a closed connected subset
    $C \subseteq \Gamma$ is a \emph{local maximum} of $f$ if $f$
    is constant on $C$ and there exists some open neighbourhood $U$ of $C$, 
    with $f(U\setminus C) < f(C)$ 
    (in the sense that for all $x \in U\setminus C$, $y \in C$, we have
    $f(x) < f(y)$).
\end{definition}

\begin{proposition}
    \label{prop:reduced-local-max}
    Let $D \in \Div(\Gamma)$ be a $v$-reduced divisor.
    Then for all $f \in R(D)$, if $Z \subseteq \Gamma$ is a local maximum for
    $f$, then $v \in Z$.
\end{proposition}

\begin{proof}
    This will follow if we show that $Z$ can fire with respect to $D$.
    Since $Z$ is a local maximum, $f$ has strictly negative integral slope along
    all outgoing tangents on $\boundary Z$. It follows that for any 
    $x \in \boundary Z$,
    \begin{equation*}
        \ord_x(f) \leq -\degout_Z(x),
    \end{equation*}
    and so since $(D+\fdiv(f))(x) \geq 0$, this implies
    $D(x) - \degout_Z(x) \geq 0$ and so $Z$ can fire.
\end{proof}

\begin{corollary}
    \label{cor:reduced-max}
    Let $D \in \Div(\Gamma)$ be a $v$-reduced divisor and
    $f \in R(D)$ a rational function. Then $f$ admits its maximum in $v$.
\end{corollary}

\begin{proof}
    Let $Z \subseteq \Gamma$ be the set on which $f$ admits its maximum, then 
    $Z$ is a local maximum of $f$ and so by Proposition 
    \ref{prop:reduced-local-max},
    $Z$ contains $v$.
\end{proof}

\begin{corollary}
    \label{cor:superlevel-connected}
    Let $D \in \Div(\Gamma)$ be a $v$-reduced divisor and
    $f \in R(D)$ a rational function. For all $a \leq f(v)$,
    the subgraph $f^{-1}([a, \infty))$ is connected.
\end{corollary}

\begin{proof}
    Suppose there was some $a$ with
    $Y= f^{-1}([a, \infty))$ disconnected. Then let $C$ be a connected component
    of $Y$ such that $v \notin C$, then if we denote $Z\subseteq \Gamma$
    the set on which 
    $f|_C$ admits its maximum, we have that $Z$ is a local maximum of $f$ and
    hence $v \in Z$, a contradiction.
\end{proof}

\begin{proposition}
    \label{prop:reduced-effective}
    Suppose $D \in \Div(\Gamma)$ is a $v$-reduced divisor. Then $D$ is linearly
    equivalent to an effective divisor if and only if $D$ is effective.
\end{proposition}

\begin{proof}
    If $D$ is effective then the statement is clear. Suppose
    $D + \fdiv(f) \geq 0$ for some $f$. Then $f$ admits its maximum in $v$
    by Corollary \ref{cor:reduced-max}. In particular, the outgoing slopes of
    $f$ at $v$ are all negative (or zero). This implies that
    $\ord_v(f) \leq 0$ and hence
    \begin{equation*}
        D(v) \geq D(v) + \ord_v(f) = (D + \fdiv(f))(v) \geq 0,
    \end{equation*}
    so $D$ is effective.
\end{proof}

\subsection{Tropical modules}
\label{sec:tropical-modules}

We will now discuss tropical modules, which are a natural structure that
appear in the context of linear systems. We build on the discussion in
\cite[Section 3]{linsys}.

\begin{definition}
    The \emph{tropical semifield} $(\R\cup\{-\infty\}, \oplus, \odot)$
    is the set of real numbers $\R \cup \{-\infty\}$ with infinity
    with the two tropical operations defined by
    \begin{align*}
        a \oplus b &= \max(a, b),\\
        a \odot b &= a + b.
    \end{align*}
    A \emph{tropical module} $(M, \oplus, \odot, -\infty)$
    is a semi-module over the tropical semi-ring.

    For any set $E$, the space $\R^E\cup\{-\infty\}$
    is naturally a tropical module.
    Clearly, $\PL(\Gamma)$ and $\Rat(\Gamma)$ are stable under tropical addition
    and scaling, so we have an inclusion of tropical modules
    $\R^\Gamma\cup\{-\infty\} \supset \PL(\Gamma) \supset \Rat(\Gamma)$,
    where by abuse of notation we implicitly consider $-\infty$ to be part 
    of these tropical modules.
\end{definition}

\begin{proposition}
    Let $D$ be an effective divisor.
    Then the space
    $R(D)$ with the point-wise tropical operations is a tropical module.
\end{proposition}

\begin{proof}
    For any $c \in \R$ and $f \in R(D)$, we have that $\fdiv(c \odot f)
    = \fdiv(f)$
    as adding a constant does not change the slopes of $f$, so clearly
    $R(D)$ is stable under tropical scaling.

    Let $f, g \in R(D)$ and $x\in \Gamma$. If $f(x) \neq g(x)$, w.l.o.g.
    $f(x) > g(x)$, then for all tangent
    vectors $\zeta \in T_x\Gamma$, $s_\zeta(f \oplus g) = s_\zeta(f)$.
    It follows that $\ord_x(f\oplus g) = \ord_x(f)$
    and so 
    \begin{equation*}
        (D + \fdiv(f \oplus g))(x) = (D + \fdiv(f))(x) \geq 0
    \end{equation*}
    If instead $f(x) = g(x)$, then for all tangent vectors
    $\zeta \in T_x\Gamma$,
    \begin{equation*}
        s_\zeta(f \oplus g) = \max(s_\zeta(f), s_\zeta(g)) \geq s_\zeta(f),
    \end{equation*}
    and so in particular $\ord_x(f \oplus g) \geq \ord_x(f)$.
    We deduce that
    \begin{equation*}
        (D + \fdiv(f\oplus g))(x) \geq (D + \fdiv(f))(x) \geq 0,
    \end{equation*}
    and so we conclude that
    $f\oplus g \in R(D)$.

    Hence $R(D)$ is stable under tropical addition and scaling, so it is a
    tropical submodule of $\Rat(\Gamma)$.
\end{proof}

As seen before, we have that $|D| \cong R(D)/\R$ and this gives $|D|$ additional
structure we can work with.

\begin{definition}
    Let $M$ a tropical module, we define the \emph{tropical
    projectivization} of $M$ to be
    \begin{equation*}
        \T(M) = (M\setminus\{-\infty\})/\R,
    \end{equation*}
    which is the quotient of $M \setminus \{-\infty\}$ modulo tropical scaling.
    We call such a space a \emph{tropical projective space}.
\end{definition}

\begin{definition}
    Let $X = \T(M)$ be a tropical projective space. We say a subset 
    $Y \subset X$ is \emph{tropically convex}, if $Y = \T(N)$ where $N$ is a
    tropical submodule of $M$.
\end{definition}

\begin{remark}
    When $D \sim D'$, then $D' = D + \fdiv(f)$, so the tropical modules
    $R(D)$ and $R(D')$ are isomorphic via the mapping
    \begin{align*}
        \phi: R(D) &\to R(D')\\
        g &\mapsto g - f.
    \end{align*}
    This is clearly a morphism as for $g, h \in R(D)$,
    \begin{align*}
        \phi(g\oplus h) &= \phi(\max(g, h))\\
        &= \max(g, h) - f\\
        &= \max(g - f, h-f)\\
        &= \phi(g)\oplus\phi(h)
    \end{align*}
    and for $c \in \R \cup \{-\infty\}$,
    \begin{equation*}
        \phi(c\odot g) = c + g - f = c\odot\phi(g).
    \end{equation*} 
    We deduce that the tropical projective space structure on $|D|$ does not
    depend on the chosen divisor $D$.
\end{remark}

We will now define a binary operator useful for studying tropical modules. This
operator has been used in the implementation to check whether a rational
function belongs to the submodule spanned by a chosen set of generators.

\begin{definition}
    Let $M \subseteq \R^E\cup\{-\infty\}$ be a tropical submodule. We define the
    binary operator $\inner{\cdot}{\cdot}$ by
    \begin{equation*}
        \inner{f}{g} = \inf_{x\in E}\{f(x) - g(x)\}
    \end{equation*}
    for all $f, g \in M$, and $g \neq -\infty$.

    For $f, g \in M$ we say $\inner{f}{g}\odot g$ is the \emph{projection}
    of $g$ on $f$.
\end{definition}

\begin{remark}
    When $M$ is a submodule of 
    $\PL(\Gamma) \subseteq \R^\Gamma \cup \{-\infty\}$,
    or when $E$ is finite, the infimum
    is attained.
\end{remark}

\begin{remark}
    \label{rmk:proj-ineq}
    We have that
    \begin{equation*}
        (\inner{f}{g}\odot g)(x) = \inf_{y \in E}\{f(y) - g(y)\} + g(x)
        \leq f(x)
    \end{equation*}
    for all $x \in E$, and so $\inner{f}{g} \odot g \leq f$.
\end{remark}

\begin{proposition}
    Let 
    $M \subseteq \R^E \cup \{-\infty\}$ be a finitely generated submodule and
    $G \subseteq M$ a finite generating set. Then for all
    $f \in M$,
    \begin{equation*}
        f = \bigoplus_{g \in G}\inner{f}{g}\odot g.
    \end{equation*} 
\end{proposition}

\begin{proof}
    By Remark \ref{rmk:proj-ineq}, we already know that
    \begin{equation*}
        f \geq \bigoplus_{g \in G}\inner{f}{g}\odot g.
    \end{equation*}

    Since $G$ is a generating set for $M$, there exist some $a_g \in
    \R\cup\{-\infty\}$ for all $g \in G$ such that
    \begin{equation*}
        f = \bigoplus_{g \in G}a_g\odot g.
    \end{equation*}  
    Now, for all $g$, we have that $f \geq a_g \odot g$. It follows
    that $f - g \geq a_g$ and so $\inner{f}{g} \geq a_g$.
    We deduce that
    \begin{equation*}
        \bigoplus_{g \in G}\inner{f}{g}\odot g \geq
        \bigoplus_{g \in G}a_g\odot g = f,
    \end{equation*}
    which shows the other inequality.
\end{proof}

\begin{corollary}
    \label{cor:span-inner}
    Let $M \subseteq \R^E \cup \{-\infty\}$ a tropical submodule and
    $G$ a finite subset. Then $f \in M$ belongs to the submodule spanned by $G$
    if and only if
    \begin{equation*}
        f = \bigoplus_{g \in G}\inner{f}{g}\odot g.
    \end{equation*} 
\end{corollary}

These properties are quite useful for studying tropical modules when we know
their set of generators. Unfortunately, it is in general not
easy to find a set of
generators of a tropical submodule.
Luckily, when $M$ is a finitely generated
submodule of $\R^E \cup \{-\infty\}$, we can give an explicit characterization
of the elements
belonging to a minimal generating set and in some cases it is even
possible to find
these exhaustively.

\begin{definition}
    An element $f \in M$ is called \emph{extremal} if for any 
    $g,h \in M$ such that $f = g \oplus h$, it holds that either
    $f = g$ or $f = h$.
\end{definition}

\begin{remark}
    An element $f \in M$ is extremal if and only if for all $c \in \R$,
    $c\odot f$ is extremal as well.
\end{remark}

\begin{proposition}\cite[Proposition 8]{linsys}
    Let $M \subseteq \R^E \cup \{-\infty\}$ a finitely generated
    tropical module.
    The set of extremals of $M$ is finite (up to tropical scalar
    multiplication),
    and is a minimal generating set of
    of $M$.
\end{proposition}

\begin{proposition}\cite[Theorem 6]{linsys}
    The tropical semi-module $R(D)$ finitely generated.
\end{proposition}

The description of the generating set of a tropical module in terms of its
extremals is especially useful when working with $R(D)$, and in this case we may
even find a generating set explicitly.

\begin{proposition}\cite[Lemma 5]{linsys}
    \label{prop:char-extremal}
    A rational function $f \in R(D)$ is an extremal of $R(D)$
    if and only if there are no 
    proper subgraphs $\Gamma_1, \Gamma_2$ covering $\Gamma$ (in the sense 
    that $\Gamma_1 \cup \Gamma_2 =
    \Gamma$), such that each can fire on $D + \fdiv(f)$.
\end{proposition}

\begin{definition}
    We say a finite subset $A \subset \Gamma$ is a \emph{cut set}
    if $\Gamma \setminus A$ is not connected and w
    e denote by $\mathcal{S}$ the set of rational functions $f \in R(D)$,
    such that $\supp(D + \fdiv(f))_E$ is not a cut set.
\end{definition}

\begin{proposition}\cite[Theorem 6(a)]{linsys}
    The set $\mathcal{S}$ is finite (up to tropical scalar multiplication)
    and contains the set of extremals of $R(D)$.
\end{proposition}

\begin{remark}
    Since Proposition
    \ref{prop:char-extremal} gives us a way to check whether
    a function is extremal
    in a finite number of steps, this yields an algorithm that finds a minimal
    generating set of $R(D)$ in finite time.
\end{remark}

\begin{remark}
One could also check that a function $f \in \mathcal{S}$ is extremal directly
from the definition using Corollary \ref{cor:span-inner},
since $\mathcal{S}$ is a finite generating set. Indeed, $f$ is extremal if and
only if it is not in the span of $\mathcal{S} \setminus \{f\}$.
\end{remark}

\subsection{Abstract polyhedral complex structure}
\label{sec:polyhedral-complex}

Let $X$ be a Hausdorff topological space $X$. An
\emph{abstract polyhedral complex} structure on $X$ is a finite set $\cK$
of subspaces $\sigma \subseteq X$, with partial order given by inclusion,
such that
\begin{itemize}
    \item Each $\sigma \in \cK$ is homeomorphic to a convex
        rational polyhedron $|\sigma|$.
    \item The homeomorphism $\sigma \cong |\sigma|$ induces an isomorphism
        between the posets
        \[\cK_{\leq \sigma} := \{\tau \in \cK \mid \tau \subseteq \sigma\}\]
        and the poset of faces of $|\sigma|$, denoted by $F_\sigma$.
    \item When $\tau \subseteq \sigma$ is a face
        corresponding to $\eta \in F_\sigma$,
        $|\tau|$ and $\eta$ are isometric.
    \item For $\sigma, \tau \in \cK$, we have that $\sigma \cap \tau \in \cK$.
\end{itemize}

As described in \cite{amini-reduced}, a complete linear system $|D|$ 
admits
a natural structure of an abstract polyhedral complex. We will briefly
summarize how it is characterized, but we will avoid going into the technical
details.

A model $G = (V, E)$ 
of $\Gamma$ induces an abstract polyhedral complex structure on it.
Indeed, if we start with $V$
with the discrete topology, $\Gamma$ is obtained by gluing for each edge $e \in
E$ a closed interval of length $l(e)$ to the corresponding vertices.
There is then a naturally induced abstract polyhedral complex structure
on the product $\Gamma^d$.
When $G$ has no self-loops, we also get an
induced abstract polyhedral complex structure on the
$d$\textsuperscript{th} symmetric product $\Gamma^d/S_d$,
where $S_d$ is the symmetric group acting on $\Gamma^d$.
Fix for each edge $e \in E$ a direction, making it a directed edge. This allows
us to identify an edge $e$ with the interval $[0, l(e)]$.
The relative interior of a face $\sigma$ of $\Gamma^d/S_d$ is described by maps
\begin{itemize}[itemsep=0em]
    \item $m_V: V\to \N$,
    \item $m_E: E\to \bigcup_{k=0}^\infty\N^k_{>0}$,
\end{itemize}
A point $x = (x_1, \dots, x_d)$ belongs to $\relint(\sigma)$ if and only if the
following conditions hold:
\begin{itemize}[itemsep=0em]
    \item For each vertex $v \in V$, the number of $x_i$ such that $x_i = v$
        is $m_V(v)$.
    \item For each edge $e \in E$ let $y_1, \dots, y_r$ be the points of $e$
        appearing in $(x_1, \dots, x_d)$, ordered according to the fixed
        orientation of $e$. Then the number of $x_i$ such that $x_i = y_j$
        is $m_E(e)_j$.
\end{itemize}

Fix an ordering $e_1, \dots, e_s$ of edges. 
Let $k_i$ be the length of the sequence $m_E(e_i)$.
Let $n = \sum_{e \in E} k_e$, then
$\relint(\sigma)$ may be identified with the following subset of $\R^n$.
\begin{align*}
    \relint(\sigma) \cong \{x \in \R^n \mid 0 &< x_1 < \dots < x_{k_1} <
    l(e_1),\\
    0 &< x_{k_1 + 1} < \dots < x_{k_1 + k_2} < l(e_2), \dots, \\
    0 &< x_{k_1 + \dots + k_{s-1} + 1} < \dots < x_{n} < l(e_s)\}
\end{align*}

When $D$ is a divisor of degree $d$, $|D|$ embeds into the abstract
polyhedral complex
complex $\Gamma^d/S_d$. Denote $\cK$ the poset of faces
(also called \emph{cells}) of $\Gamma^d/S_d$.
The intersection of $|D|$ with the relative
interior of a face $\tau \in \cK$ consist of
a finite number of connected components. It turns out that the set
$\cK_D$ consisting of the closures of all such connected components gives $|D|$
the structure of an abstract polyhedral complex. A face $\tau \in \cK_D$ is
by definition contained in some face $\sigma \in \cK$.
The homeomorphism $\tau \cong |\tau|$ is then given by co-restriction of the
embedding
\[\tau \hookrightarrow \sigma \cong |\sigma|.\]

Similarly as before, we may describe the
relative interior of a face $\tau \in \cK_D$ with a triple
$(m_V, m_E, s)$ given by maps
\begin{itemize}[itemsep=0em]
    \item $m_V: V\to \N$,
    \item $m_E: E\to \bigcup_{k=0}^\infty\N^k_{>0}$,
    \item $s: E \to \Z$.
\end{itemize}
A divisor $D' = D + \fdiv(f)$ then belongs to $\relint(\tau)$ if and only if the
following conditions are satisfied:
\begin{itemize}[itemsep=0em]
    \item For each vertex $v \in V$, $D'(v) = m_V(v)$.
    \item For each edge $e \in E$, 
        $D'|_e = \sum_i m_E(e)_i x_i$, where 
        $0 < x_1 < \dots < x_k < l(e)$.
    \item For each edge $e \in E$ with corresponding tangent $\zeta$ based at
        the origin of the edge, $s_\zeta(f) = s(e)$.
\end{itemize}

\begin{remark}
    Not all triples $(m_V, m_E, s)$ will correspond to a non-empty face of $|D|$.
\end{remark}

We will now define the notion of
\emph{definable} subset, following \cite{linsys-independence}.

\begin{definition}
    A subset of $\R^n$ is called \emph{definable} if it can be written as a
    finite expression involving intersections, unions, and complements
    of closed half-spaces 
    \[H_i = \{u \in \R^n \mid \inner{u}{v_i} \geq a_i\}.\]
    A subset $X \subseteq |D|$ is \emph{definable} if
    the image of $X \cap \sigma \hookrightarrow |\sigma|$ is definable for all
    $\sigma \in \cK_D$. 
\end{definition}

The reason why this notion is nice is that it allows us to define a reasonable
notion of dimension around a point for a large class of subspaces of $|D|$.

\begin{proposition}
    \label{prop:closed-definable-polyhedra}
    When $Y \subseteq \R^n$ is a closed and
    definable subset, then $Y$ is a finite union of (possibly unbounded)
    polyhedra.
\end{proposition}

\begin{proof}
    By assumption, $Y \subseteq \R^n$ is written as a finite expression
    involving intersections, unions, and complements
    of closed half-spaces $H_i$.
    We can rewrite this expression to write it under the form
    \begin{equation*}
        Y = \bigcup_{i = 1}^k\bigcap_{j = 1}^l H_{i, j},
    \end{equation*}
    where the $H_{i, j}$ are either closed, or open half-spaces.
    We call an expression of this form
    a \emph{disjunctive normal form} (it is not unique).
    We may assume that for all $i$ 
    \begin{equation*}
        \bigcap_{j = 1}^l H_{i, j}
    \end{equation*}
    is non-empty, as otherwise we can just remove this term from the expression.
    
    However as we assumed that $Y$ is closed, we have that
    \begin{equation*}
        Y = \closure{Y} = \bigcup_{i = 1}^k\closure{\bigcap_{j = 1}^l H_{i, j}} 
        = \bigcup_{i = 1}^k\bigcap_{j = 1}^l \closure{H_{i, j}}.
    \end{equation*}
    To see the last equality, note that
    \[\closure{\bigcap_{j = 1}^l H_{i, j}} \subseteq
    \bigcap_{j = 1}^l \closure{H_{i, j}}. \]
    Take any $x \in \bigcap_{j = 1}^l \closure{H_{i, j}}$ and
    $y \in \bigcap_{j= 1}^l H_{i, j}$
    (such an $y$ exists as we assumed the intersection
    is non-empty). Then define 
    \[x_n = \frac{1}{n}y + \left(1 - \frac{1}{n}\right)x.\]
    We will show that this sequence is contained in $\bigcap_{j= 1}^l H_{i, j}$.
    Suppose for the sake of contradiction that there are some $i, j, n$
    such that $x_n \notin H_{i, j}$. If $H_{i, j}$ was a closed half-space, then
    $x \in H_{i, j}$, but this would imply that
    $[x, y] \subseteq H_{i, j}$ as $H_{i, j}$
    is convex. So $H_{i, j}$ has to be an open half-space. Write
    \[H_{i, j} = \{u \in \R^n \mid \inner{u}{v} < a\},\]
    then since $x \in \closure{H_{i, j}}$, it follows that $\inner{x}{v} \leq
    a$. Furthermore $\inner{y}{v} < a$. But by bilinearity we get that
    \[\inner{x_n}{v} = \frac{1}{n}\inner{y}{v} + \left(1 - \frac{1}{n}\right)
    \inner{x}{v} < a.\]
    This shows that $x_n \in H_{i, j}$, a contradiction.

    This shows that the sequence of $x_n$ is contained in
    $\bigcap_{j=1}^l H_{i, j}$, and hence since $x$ is 
    the limit of this sequence, it lies in
    $\closure{\bigcap_{j = 1}^l H_{i, j}}$. This shows the other
    inclusion.

    As a result, we may assume that all the $H_{i, j}$ are actually closed
    half-spaces and so this concludes the proof.
\end{proof}

\begin{corollary}
    Let $Y \subseteq \R^n$ a closed and definable subset.
    Then $Y$ admits the structure of a polyhedral complex.
\end{corollary}

\begin{proof}
    Since $Y$ is a finite union of polyhedra,
    up to subdividing the polyhedra we
    may assume that the intersection of any two given polyhedra is a face of
    each respective polyhedron. It follows that $Y$ admits a natural structure
    of polyhedral complex.
\end{proof}

\begin{corollary}
    \label{cor:closed-def-complex}
    If $X \subseteq |D|$ is a closed and definable subset, then
    $X$ admits the structure of an abstract polyhedral complex.
\end{corollary}

\begin{proof}
    For all $\sigma \in \cK_D$, $X\cap \sigma$ admits the structure of
    a polyhedral complex $\cK_\sigma$,
    and up to refining the polyhedral complex structure,
    we may assume that for any 
    $\tau \subseteq \sigma$,
    $\cK_\sigma$ restricted to faces contained in $\tau$
    agrees with $\cK_\tau$. The abstract polyhedral
    complex structure $\cK_X$ on $X$ is then given by the union of the
    $\cK_\sigma$, where for $\tau \subseteq \sigma$ we identify $\cK_\tau$
    with the corresponding subset of $\cK_\sigma$.
\end{proof}

\begin{definition}
    Let $X$ be an abstract polyhedral complex,
    and $x \in X$. We define the
    dimension of $X$ at $x$ to be
    \begin{equation*}
        \dim_x X := \max\{\dim(\sigma) \mid \sigma \in \cK_X, x \in \sigma\}.
    \end{equation*}
    Here $\dim(\sigma)$ denotes the dimension of the smallest affine subspace
    containing $|\sigma|$.
\end{definition}

\begin{definition}
    Let $X$ be an abstract polyhedral complex. 
    We define the relative
    interior of $X$ to be
    \begin{equation*}
        \relint(X) := \bigcup_{\substack{\sigma \in \cK_X\\
        \sigma\textrm{ maximal}}} \relint(\sigma).
    \end{equation*}
\end{definition}

\begin{remark}
    Let $X$ be a closed definable subset of $|D|$, equipped with the induced
    abstract polyhedral complex structure. Let $D' \in X$ be a divisor, which
    belongs 
    to the relative interior of a unique
    face $\sigma$ of $\Gamma^d/S_d$ and also of a unique face $\tau$ of $X$.
    Then $\tau \subseteq \sigma$ and in addition
    $\tau$ embeds as a polytope in $\sigma$, so we get naturally
    an embedding of manifolds $\relint(\tau) \hookrightarrow \relint(\sigma)$.
    If $D' \in \relint(X)$, then $\tau$ is a maximal face of $X$,
    and so $\tau$ has dimension $\dim_x X$.
    In particular, this implies that we may detect the
    dimension of $X$ at $D'$ from the dimension of the
    tangent space of $\relint(\tau)$ at $D'$,
    which we may naturally identify as a subspace of the tangent space of
    $\relint(\sigma)$ at $D'$.

    Recall that $\relint(\sigma)$ can be naturally identified as a subspace of
    $\R^n$ for some $n$, with coordinates corresponding to the positions of
    the chips of $D'$ along the edge they are supported on.
    A tangent vector at $D'$ may be therefore coordinatized by the rate at which
    it moves each of the chips of $D'$ in a given direction along each edge.
    In other words, it corresponds to an infinitesimal transformation of
    $D'$ and may be represented by a continuously differentiable
    path $[0, \epsilon) \to \relint(\sigma)$
\end{remark}

The above discussion motivates a definition of tangent space of an arbitrary
subspace of $\Gamma^d/S_d$ at a given point.

\begin{definition}
    Let $X \subseteq \Gamma^d/S_d$ be a subspace. We define the tangent space of
    $X$ at $x$, denoted by $T_xX$ as the linear subspace of $T_x\Gamma^d/S_d$
    generated by the tangent vectors that may be represented as a continuously
    differentiable path $[0, \epsilon) \to \Gamma^d/S_d$ whose image is
    contained in $X$.
\end{definition}

\begin{remark}
    \label{rmk:dims-tangents}
    When $X$ is an abstract polyhedral complex and $x \in \relint(X)$, then
    \begin{equation*}
        \dim_xX = \dim T_xX.
    \end{equation*}
\end{remark}

\subsection{Structure of complete linear systems}
\label{sec:complete-linear-systems}

In what follows we fix a model $G = (V, E)$
of $\Gamma$ without self-loops.
As discussed previously, a complete linear system $|D_0|$
inherits the structure of
an abstract polyhedral complex which depends on this model.

\begin{definition}
    Let $D \in |D_0|$ be a divisor.
    We say $D$ is \emph{generic} if $D \in \relint(|D_0|)$.
\end{definition}

\begin{remark}
    For any divisor $D \in |D_0|$,
    there is a unique cell $\Delta_D$ of $|D_0|$ such that
    $D$ belongs to the relative interior of $\Delta_D$. It follows that
    $D$ is generic if and only if $\Delta_D$ is an inclusion-wise maximal cell.
\end{remark}

\begin{remark}
    The set of generic divisors of $|D_0|$ is dense in $|D_0|$.
\end{remark}

For $D$ a divisor, we will split it as a sum of the divisors
$D_V := D|_V$ and $D_E := D|_{\Gamma \setminus V}$.

\begin{proposition}
    \cite[Prop. 13]{linsys}
    \label{prop:dim-comp}
    Let $D \in |D_0|$ be a divisor, then
    \begin{equation*}
        \dim \Delta_D = \#\{\textrm{connected components of }
        \Gamma\setminus \supp{D_E}\} - 1
    \end{equation*}
\end{proposition}

\begin{definition}
    Let $A$ be a finite subset of $\Gamma$. We say a divisor
    $D \in \Div(\Gamma)$ is \emph{$A$-unsaturated} whenever
    there is no closed subgraph $C \subseteq \Gamma$ with 
    $\boundary{C} \cap A \neq \emptyset$ that can fire.
\end{definition}

\begin{remark}
    We are mostly interested in the case where the set $A$ is the set of
    vertices $V$ of our model $G$.
\end{remark}

\begin{proposition}
    \label{prop:generic-char}
    Let $D$ be a divisor. Then $D$ is generic if and only if the following
    conditions are satisfied:
    \begin{enumerate}[1)]
        \item For all $x \in \Gamma$, we have that $D(x) < \val(x)$.
        \item $D$ is $V$-unsaturated
    \end{enumerate}    
\end{proposition}

\begin{proof}
    Suppose $D$ is generic, then it follows immediately from Proposition 
    \ref{prop:dim-comp} that
    there is no subgraph $C$ that can fire,
    that satisfies
    \begin{equation*}
        \#\supp(\fdiv(D + CF(Z, \epsilon))_E) > \#\supp D_E
    \end{equation*}
    for all $\epsilon > 0$ small enough.

    If for some
    $x \in \Gamma$ we had that
    $D(x) \geq \val(x)$, then taking $C = \{x\}$,
    we would get that the subgraph $C$ can fire, and by firing it we 
    would get a
    divisor with more points supported on the edges.
    So condition 1) is satisfied.

    If there was a closed subgraph $C \subseteq \Gamma$ with
    $\boundary C \cap V
    \neq \emptyset$, then firing $C$ would move at least one
    chip from a vertex to the
    interior of an edge. Again, this would yield a divisor with more points
    supported on the edges. Hence $D$ is $V$-unsaturated and so condition 2) is
    satisfied as well.

    Now, suppose $D$ satisfies the two conditions. For the sake of
    contradiction, suppose $D$ is not generic, then for all $\epsilon > 0$,
    there exists some $D' = D + \fdiv(f) \in |D|$, such that
    $\|f\|_\infty \leq \epsilon$ and $D'$
    belongs to some higher dimensional cell.
    By Lemma \ref{lemma:tropical-sum}, we may write
    $f = f_1 + \dots + f_n$, where the $f_i = CF(Z_i, \epsilon_i)$
    are chip firing moves with 
    $\epsilon_i = \|f_i\|_\infty \leq \|f\|_\infty$. Furthermore, the Lemma also
    ensures that
    $f_1 \in R(D)$. By condition 1), $Z_1$
    contains no isolated point, and since $D$ is $V$-unsaturated,
    $\boundary Z_1$ does not contain
    any vertices. We deduce that for $\epsilon$
    small enough, $D + \fdiv(f_1)$
    has the same combinatorial type as $D$. In addition, $f_1$ is constant in
    a neighbourhood of each vertex, so we deduce that $D + \fdiv(f_1)$
    belongs to the relative interior of the 
    same cell of $|D|$ as $D$. We can repeat the argument to get by
    induction that $D'$ belongs to the relative interior of the 
    same cell as $D$, a contradiction.
\end{proof}

\begin{figure}
    \centering
    \begin{tikzpicture}[domain=0:4,
        declare function={
            f(\x) = -abs(\x-2)+2 + \x + abs(\x - 3.5) - 3.5;
        },
    ]
        \draw[color=red]    plot (\x,{f(\x)})    node[right]
          {$f$};
        \draw[color=gray, dashed]   plot (\x,1.5)    node[right] {};
        \draw[color=black, thick, samples=100] plot (\x,{max(f(\x), 1.5) +
          0.03}) node[right] {$f_t$};
      \draw [|<->|] (-0.2, 1.5)-- node[left=1mm] {$t$} (-0.2, 2);
    \end{tikzpicture}
    \caption{Example of $f$ and $f_t = f\oplus(\sup(f) - t)$}
    \label{fig:f_t}
\end{figure}
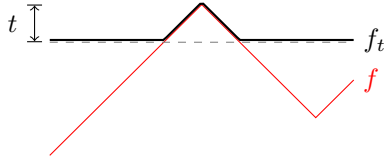

\begin{proposition}
    \label{prop:rank-dim}
    For any generic divisor $D \in |D_0|$, we have that
    \[\dim \Delta_D \geq r(D).\]
\end{proposition}

\begin{proof}
    We will show by induction on $r$ that if $r(D) \geq r$, then 
    $\dim \Delta_D \geq r$, which implies the result. 
    The base case $r = 0$ is trivially verified, so suppose $r \geq 1$.

    Since $r(D) \geq r \geq 1$,
    there exists some non-constant 
    rational function $f \in R(D)$, then define
    \begin{equation*}
        f_t = f \oplus (\sup(f) - t).
    \end{equation*}
    See Figure \ref{fig:f_t} for a depiction of $f_t$ on an interval.
    The map $t \in [0, \epsilon) \mapsto D + \fdiv(f_t)$ determines a
    tangent vector $\zeta$ in $|D|$.

    Let $S := V \cup \bend(f)$, then $f(S)$ is a finite set and so 
    for $\epsilon > 0$ small enough, the set
    \[f(S) \cap [\sup(f) - \epsilon, \infty)\]
    will consist only of a single point, which implies that $f_\epsilon$ will be 
    a weighted chip-firing move (and so will be all the $f_t$ for $0 < t <
    \epsilon$).
    Let $Z$ be the subgraph on which $f_\epsilon$ attains its maximum, then $Z$
    can fire as 
    \[D(x) \geq -\fdiv(f_\epsilon)(x) \geq \degout_Z(x).\]
    Since $D$ was assumed to be generic, by Proposition \ref{prop:generic-char}
    we know that $Z$ contains no isolated points (as otherwise for such an $x$ we
    would get $\degout_Z(x) = \val(x)$ and so $D(x) \geq \val(x)$), and
    $\boundary Z \cap V = \emptyset$, as $D$ is $V$-unsaturated.
    In particular, $\boundary Z \cap \supp D$ consists only of points on the
    interior of edges, and for such $x$, we know that $D(x) < \val(x) = 2$.
    We deduce that $f_\epsilon$ can only be an ordinary chip firing move
    $CF(Z, \epsilon)$. Now, take any $x \in \boundary Z \cap \supp D$.
    Clearly, 
    \[r(D - x) \geq r(D) - 1 \geq r - 1.\]
    Furthermore, $D - x$ is
    generic as it clearly satisfies the conditions of
    Proposition \ref{prop:generic-char}.
    By the induction hypothesis, we obtain that $\Delta_{D - x} \geq r - 1$
    and so let $\zeta_1, \dots, \zeta_{n-1}$ be linearly independent tangent
    vectors
    at $D - x$.
    Since $\Delta_{D - x}$ naturally embeds into $\Delta_{D}$ via the
    map $D' \mapsto D' + x$, we may see the $\zeta_i$ as vectors in the tangent
    space at $D$.
    
    We claim that the vectors $\zeta, \zeta_1, \dots, \zeta_{n-1}$ are linearly
    independent, which will imply that the tangent space at $D$ has dimension at
    least $n$ and so $\dim \Delta_D \geq n$. To see this, notice that the 
    $\zeta_i$ correspond to infinitesimal transformations of $D$ that fix the
    chip at $x$. This follows from the fact that $D(x) = 1$ and so $D - x$ has
    no chip at $x$.
    However, $x$ was chosen in the boundary of $Z$, so $\zeta$ has
    a non-zero component along the coordinate corresponding to the chip at $x$,
    so $\zeta$ is independent from the $\zeta_i$, and this completes the proof.
\end{proof}

\begin{corollary}
For any divisor $D \in |D_0|$,
we have that 
\[\dim_D|D_0| \geq r(D_0).\]
\end{corollary}

\begin{proof}
    For any maximal cell $\sigma \in \cK_{D_0}$,
    its relative interior is non-empty and so
    there exists some $D' \in \relint(\sigma)$. Note that this implies
    $\Delta_{D'} = \sigma$.
    In particular,
    \[\dim \Delta_{D'} \geq r(D') = r(D_0)\]
    Since the dimension of $|D_0|$ at $D$ is at least the dimension of a maximal
    cell containing $D$, the statement follows.
\end{proof}

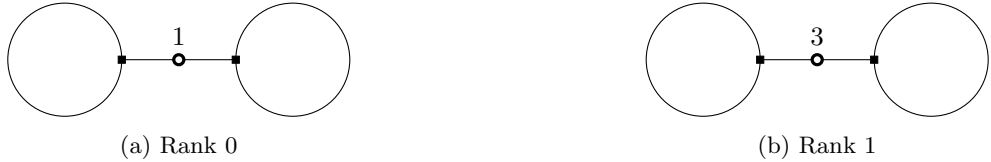
\begin{figure}[ht]
    \begin{subfigure}[h]{0.4\linewidth}
        \centering
        \begin{tikzpicture}[scale=1.5]
            \draw (-0.5, 0) circle (0.5);
            \draw (0,0) -- (1,0);
            \draw (1.5, 0) circle (0.5);
            \node[divisor, label=1] at (0.5, 0) {};
            \node[vertex] at (0, 0) {};
            \node[vertex] at (1, 0) {};
        \end{tikzpicture}
        \caption{Rank 0}
    \end{subfigure}
    \hfill
    \begin{subfigure}[h]{0.4\linewidth}
        \centering
        \label{fig:dumbbell-rank-1}
        \begin{tikzpicture}[scale=1.5]
            \draw (-0.5, 0) circle (0.5);
            \draw (0,0) -- (1,0);
            \draw (1.5, 0) circle (0.5);
            \node[divisor, label=3] at (0.5, 0) {};
            \node[vertex] at (0, 0) {};
            \node[vertex] at (1, 0) {};
        \end{tikzpicture}
        \caption{Rank 1}
    \end{subfigure}

    \caption{Divisors whose complete linear system is locally of dimension
        strictly greater than the rank.}
    \label{fig:dumbbell-rank}
\end{figure}

\begin{remark}
    One may wonder whether the rank is fully encoded in the dimension of the
    complete linear system. In general this is not true, as it might
    happen that $|D|$ only has cells of dimension strictly higher than
    $r(D)$. The dumbbell graph provides simple counter-examples,
    as illustrated in the Figure
    \ref{fig:dumbbell-rank}. The complete linear system corresponding to the
    divisor in (a) is just an
    interval, so equi-dimensional of dimension 1. The complete linear system 
    corresponding to the divisor in (b) has maximal cells of dimensions at least
    $2$. Indeed, the maximal cells are of three sorts:
    \begin{itemize}[itemsep=0em]
        \item All three chips on the bridge (dimension $3$)
        \item Two chips on one circle and the third on the bridge (dimension
            $2$)
        \item All three chips on the same circle (dimension $2$)
    \end{itemize}
\end{remark}

\subsection{Structure of the canonical linear system}
\label{sec:structure-canonical}

We will now study the case of the canonical linear system. We recall that the
canonical divisor is defined by
\[K = \sum_{x \in \Gamma}(\val(x) - 2) \cdot x.\]
The canonical linear system depends tightly on the cycles present in the metric
graph as we will soon see. We will start with a few lemmas.

\begin{lemma}
    \label{lemma:conn-comp-genus}
    Let $D$ be a $V$-unsaturated divisor.
    Then for any connected component $C$ of $\Gamma
    \setminus \supp D_E$, we have that $g(C) \geq \deg D|_C$
\end{lemma}

\begin{proof}
    We proceed by induction on $\deg D|_C$.
    The fact is clearly true when $\deg D|_C = 0$, as $g(C)$ is always
    non-negative. So suppose the Lemma holds when $\deg D|_C < n$ for some
    fixed $n \in \N$, we will now show that it also holds when
    $\deg D|_C = n$.

    We may see $D|_C$ as a divisor in $\hat{C}$.
    Fix some $x \in C \setminus \supp D|_C$ and let $A$ be the connected
    component of $x$ in $\hat{C} \setminus \supp D|_C$ and let $B$
    be the complement of $A$.
    We claim that there is some $y \in \supp D|_C$ such that 
    $\degout_{B}(y) > D(y)$.
    For the sake of contradiction, suppose that for all $y \in \supp D|_C$,
    $\degout_{B}(y) \leq D(y)$, then we can fire $B$ in $\hat{C}$.
    Let $B'$ be the image of $B$ in $\Gamma$, then
    $\boundary B' \setminus C$
    consists of points in $\supp D_E$. Let $y$ be such a point, then
    $\val(y) = 2$, which implies that $\degout_{B'}(y) = 1$ as $y$ is on the 
    boundary of $B'$. Furthermore, $D(y) \geq 1$ since $y \in \supp D_E$.
    We deduce that $B'$ can fire in $\Gamma$.
    Since we assumed $D$ is $V$-unsaturated, this implies
    $\boundary B' \cap V = \emptyset$ and so
    $\boundary B' \subseteq \supp D_E$.
    We deduce that $B'$ is all of $\closure{C}$, but
    this would imply that $B$ is all of $\hat{C}$ (since only a finite number of
    points are identified in the gluing $\hat{C} \to \overline{C}$ and
    $B$ is closed),
    which is absurd as $A$ is non-empty and open.

    So let $y \in \supp D|_C$ such that $\degout_{B}(y) > D(y)$.
    Then choose $\zeta \in T_y\Gamma$ along some edge in
    $A$. Cut $\Gamma$ along $\zeta$ to obtain a new metric graph $\Gamma'$,
    and let $C'$ be the set of points of $\Gamma'$ lying above $C$.
    We claim that $C'$ is connected.
    Indeed, since $y \in \supp D|_C$, it follows that $D(y) \geq 1$ and so
    $\degout_{B}(y) \geq 2$. In particular, this means there are at least
    two independent paths between $x$ and $y$ in $A$,
    which go along a different
    tangent at $y$. So after the cut, one of the two paths has to stay intact,
    which implies the connectedness of $C'$.
    In $\Gamma'$ there are two points lying above $y$, say $y_0, y_1$ and
    suppose $y_1$ is the leaf that corresponds to the cut along $\zeta$.
    Lift the model of $\Gamma$ to $\Gamma'$ by letting
    \[V' = (V\setminus \{y\})\cup\{y_0, y_1\}.\]

    Define a divisor $D'$ on $\Gamma'$ by lifting $D$ to $\Gamma'$ on
    $\Gamma \setminus \{y\}$ and let $D'(y_0) = D(y) - 1$
    and $D(y_1) = 0$. We claim that $D'$ is a $V'$-unsaturated divisor of $\Gamma'$.
    Suppose there is a closed subgraph $Z' \subseteq \Gamma'$ with
    $\boundary{Z'}\cap V' \neq \emptyset$ that can fire. Then for all
    $z \in \boundary{Z'}$, we have that $\degout_{Z'}(z) \leq D'(z)$.
    Let $Z$ be the image of $Z'$ under the gluing $\Gamma' \to \Gamma$.
    It follows that for $z \in \boundary{Z}\setminus\{y\}$,
    \begin{equation*}
        \degout_{Z}(z) = \degout_{Z'}(z) \leq D'(z) = D(z).
    \end{equation*}
    If $y \in \boundary{Z}$, then we have by construction that
    \begin{equation*}
        \degout_{Z}(y) \leq \degout_{Z'}(y_0) + 1 \leq D'(y_0) + 1 = D(y).
    \end{equation*}
    We conclude that $Z$ can fire and so $\boundary{Z}\cap V = \emptyset$
    since $D$ is $V$-unsaturated. In particular, $\boundary{Z}$ is contained 
    in the image of $\boundary{Z'}$, so we deduce that $\boundary{Z'}\cap V' =
    \emptyset$ by choice of model $V'$.

    So we may apply the induction hypothesis to $\Gamma', D', C'$ to get that
    $g(C') \geq \deg D'|_{C'}$. Now, notice that $\hat{C'}$
    has the same number of
    edges as $\hat{C}$ but one extra vertex ($\hat{C'}$ is the cut of $\hat{C}$
    along $\zeta$), so 
    by Remark \ref{rem:genus-formula}, we have that $g(C') = g(C) - 1$.
    Furthermore, we have by construction of $D'$ that
    $\deg D'|_{C'} = \deg D|_C - 1$,
    and so the conclusion follows.
\end{proof}

\begin{lemma}
    \label{lemma:dim-deg-formula}
    Suppose $\Gamma$ is connected.
    Let $D$ be a generic divisor. Let $C_1, \dots, C_n$
    be the connected components of $\Gamma \setminus \supp D_E$.
    Then
    \begin{equation*}
        \dim \Delta_D = \deg D - g(\Gamma) + 
        \sum_{i = 1}^n\left(g(C_i) - \deg D|_{C_i}\right)
    \end{equation*}
\end{lemma}

\begin{proof}
    The points of $\supp D_E$ are all of valence 2, furthermore, since $D$ is
    generic, by Proposition \ref{prop:generic-char},
    $D(x) = 1$ for
    all $x \in \supp D_E$.
    So we deduce from Lemma
    \ref{lemma:glueing-genus} that
    \begin{equation*}
        g(\Gamma) = g(\Gamma \setminus \supp D_E) 
        + \deg D_E + 1 - N,
    \end{equation*}
    where $N$ is the number of connected components of $\Gamma \setminus D_E$.
    By Proposition \ref{prop:dim-comp}, we have that
    \begin{align*}
        \dim \Delta_D &= N - 1\\
        &= \deg D_E - g(\Gamma) + g(\Gamma \setminus \supp D_E) \\
        &=  \deg D - g(\Gamma) + g(\Gamma \setminus \supp D_E) - \deg D_V.
    \end{align*}  
    Now the result follows from the fact that
    \begin{equation*}
        g(\Gamma \setminus \supp D_E) = \sum_{i= 1}^n g(C_i),
        \qquad\textrm{and}\qquad
        \deg D_V = \sum_{i = 1}^n \deg D|_{C_i}.
    \end{equation*}
\end{proof}

By the Riemann-Roch theorem (Theorem \ref{thm:trop-rr}) we have that
\[r(K) = \deg(K) - g(\Gamma) + 1.\]
The degree of $K$ is $2g - 2$, which may be computed using Remark
\ref{rem:genus-formula}, and hence the rank of $K$ is $g - 1$.
By Proposition \ref{prop:rank-dim}, we know that maximal cells have dimension at
least $g - 1$. As we will now see, there is a general class of graphs for which
the canonical linear system always has cells of higher dimension.

\begin{definition}
    We say two cycles are \emph{disjoint}
    if they don't intersect.
\end{definition}

\begin{proposition}
    \label{prop:disjoint-dim}
    If $\Gamma$ has at least two disjoint cycles,
    then there is a cell $\sigma$ of $|K|$
    of dimension at least $g$.
\end{proposition}

\begin{proof}
    The union of the two cycles $Z = Z_1 \cup Z_2$
    is a subgraph where each vertex $v$ has $\val(v) - 2$ outgoing
    edges. Since $K(v) = \val(v) - 2$, we deduce $Z$ can fire and firing $Z$
    will remove all the chips from it. So fire $Z$ by a
    small amount $\epsilon > 0$ to obtain a divisor $F \in |K|$.
    Since generic divisors are dense in $|K|$, we
    can find a generic divisor $D$ close enough to $F$,
    so that $\supp D \cap Z =
    \emptyset$. It follows that $\Gamma \setminus D_E$ has at least two
    connected components containing a cycle and no point of $D_V$ (the
    connected components containing $Z_1$ and $Z_2$ respectively).
    By Proposition \ref{prop:generic-char} and 
    Lemma \ref{lemma:conn-comp-genus}, it follows that 
    \begin{equation*}
        \sum_{C} (g(C) - \deg D|_C)
        \geq 2,
    \end{equation*}
    where the sum is over the connected components of
    $\Gamma \setminus \supp D_E$.
    Since $\deg D = \deg K = 2g - 2$, the result follows from
    Lemma \ref{lemma:dim-deg-formula}.
\end{proof}

\begin{remark}
    The converse statement is not true, for example the bipartite graph on $6$
    vertices does not contain two disjoint cycles, however $|K|$ has a cell of
    dimension $5 \geq g(\Gamma) = 4$. This divisor is represented in Figure
    \ref{fig:bipartite}.
\end{remark}

\begin{figure}[ht]
    \centering
    \begin{tikzpicture}[scale=2.5]
    \draw (0,0)--(0,1);
    \draw (1,0)--(0,1);
    \draw (2,0)--(0,1);
    \draw (0,0)--(1,1);
    \draw (1,0)--(1,1);
    \draw (2,0)--(1,1);
    \draw (0,0)--(2,1);
    \draw (1,0)--(2,1);
    \draw (2,0)--
        node[pos=1/7, divisor] {}
        node[pos=2/7, divisor] {}
        node[pos=3/7, divisor] {}
        node[pos=4/7, divisor] {}
        node[pos=5/7, divisor] {}
        node[pos=6/7, divisor] {}
        (2,1);
    \node[vertex] at (0,0) {};
    \node[vertex] at (1,0) {};
    \node[vertex] at (2,0) {};
    \node[vertex] at (0,1) {};
    \node[vertex] at (1,1) {};
    \node[vertex] at (2,1) {};
    \end{tikzpicture}
    \caption{Divisor $D$ on bipartite graph on six vertices
    with $\dim \Delta_D = 5$. The points in the support of $D$ are all of
    multilpicity $1$.}
    \label{fig:bipartite}
\end{figure}
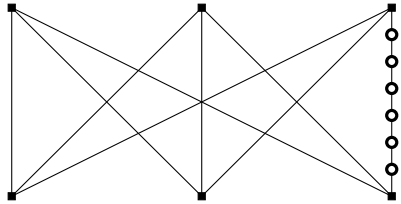

Unlike in the case of arbitrary complete linear systems, the lower bound for
dimension is always attained in the case of canonical linear systems.

\begin{proposition}
    \label{prop:minimal-dim-attained}
    For any metric graph $\Gamma$, there always exists a
    maximal cell $\sigma$ of $|K|$ of dimension $g - 1$.
    Furthermore, for any $D \in \relint(\sigma)$, we have that
    $\supp D \cap V = \emptyset$.
\end{proposition}

\begin{proof}
    By the tropical Riemann-Roch theorem (Theorem \ref{thm:trop-rr}),
    we get that $r(K) = g-1$.
    So choose $P_1, \dots, P_{g-1}$ in $\Gamma \setminus V$, such that
    $\Gamma \setminus \{P_1, \dots, P_{g-1}\}$ has genus $1$.
    By the definition of rank, there exists some divisor 
    $R \in |K - P_1 - \dots - P_{g - 1}|$, so let $S = R + P_1 + \dots +
    P_{g-1}$.

    Let $D \in |K|$ be a generic divisor
    sufficiently
    close enough to $S$, so that $B(P_i, \epsilon) \cap \supp D$ is
    non-empty for some small $\epsilon$ (we may always find such a divisor $D$,
    since generic divisors are dense in the complete linear system).
    For such a $D$ we necessarily have 
    $g(\Gamma \setminus \supp D_E) \leq 1$.

    Now we may apply
    Lemma \ref{lemma:glueing-genus} to get
    \begin{align*}
        \#\{\textrm{conn. comp. of } \Gamma \setminus D_E\} &= 
        1 + \deg D_E - (g(\Gamma) - g(\Gamma \setminus D_E))\\
        &\leq 1 + \deg D - (g - 1) = g.
    \end{align*}
    By Proposition \ref{prop:dim-comp}, we get that $\dim \Delta_D \leq g-1$.
    We also have that $\dim \Delta_D \geq g-1$,
    by Proposition \ref{prop:rank-dim},
    so we must get $\dim \Delta_D = g - 1$.
    In particular, this forces $\deg D_E = \deg D$ and so $D_E = D$.
\end{proof}

\subsection{Tropical linear systems}
\label{sec:tropical-linear-systems}

We would now like to define the notion of linear systems on metric graphs. In
algebraic geometry linear systems are projective subspaces of the complete
linear system, so in analogy we will call tropical linear system the
projectivizations of tropical submodules of $R(D_0)$.

\begin{definition} 
    Let $\fd$ be a tropically convex subset of the complete linear system
    $|D_0|$,
    we say $\fd$ is a \emph{tropical linear system}
    (or tropical linear \emph{series}).

    For any $D \in |D_0|$, we
    will denote by $R(\fd, D)$ the cone over $\fd$ in $R(D)$. In other words,
    \[R(\fd, D) := \{f \in R(D) \mid D + \fdiv(f) \in \fd\}\]
    is the tropical submodule of $R(D)$ whose tropical projectivization
    is $\fd$.
\end{definition}

\begin{remark}
    Note that $R(\fd, D)$ contains the constant functions if and only if $D \in
    \fd$.
\end{remark}

\begin{remark}
    Some authors choose to start with a tropical
    submodule $\Sigma \subseteq R(D_0)$ and denote the associated tropical linear
    system by 
    \[|\Sigma| := \T(\Sigma) \subseteq |D_0|.\]
    This is the convention used in \cite{linsys-independence} and
    \cite{kodaira-dimensions}. Note that in these two papers the authors use the
    term
    ``tropical linear series" for $\Sigma$ rather than $|\Sigma|$.
\end{remark}

\begin{definition}
    Let $E$ be an effective divisor, we define
    \begin{equation*}
        \fd(-E) = \{D \in \fd \mid D - E \geq 0\}.
    \end{equation*}
    Note that $\fd(-E)$ is also a tropical linear series as for any choice of $D
    \in |D_0|$,
    \begin{equation*}
        R(\fd(-E), D) = R(\fd, D) \cap R(D - E) \subseteq R(D), 
    \end{equation*}
    is a tropical submodule.
\end{definition}

\begin{definition} 
    We define the \emph{rank} of $\fd$ to be the integer
    \begin{equation*}
        r(\fd) = \max\{d \in \N \mid \fd(-E) \neq \emptyset
        \textrm{ for all effective divisor $E$ of degree $d$}\},
    \end{equation*}
    where if $\fd(-E) = \emptyset$, we set $r(\fd) = -1$.
\end{definition}

\begin{definition}
    We will say that a
    tropical linear system $\fd \subseteq |D_0|$ is finitely generated if
    $R(\fd, D_0)$ is finitely generated as a tropical module.
\end{definition}

\begin{proposition}
    \cite[Lemma 2.8]{linsys-independence}
    \label{prop:fin-gen-definable}
    If $\fd \subseteq |D_0|$ is a finitely generated tropical linear series,
    then
    $\fd$ is a closed, definable subset of $|D_0|$.
\end{proposition}

\begin{corollary}
    Any finitely generated tropical linear series $\fd$ admits the structure of
    an abstract polyhedral complex.
\end{corollary}

\begin{proof}
    This follows from Corollary \ref{cor:closed-def-complex}.
\end{proof}

We would now like to generalize
Proposition \ref{prop:rank-dim} to 
tropical linear systems. The problem we face is that the notion of
\emph{generic} divisor we defined in the previous section is not as tractable in
this setting, as the abstract polyhedral complex structure is not so easy to
characterize. However, we can show the result for a dense class of divisors,
which will then imply the dimension bound for all maximal cells.
Fix $\fd \subseteq |D_0|$ a complete linear system.

\begin{definition}
    Let $D \in \fd$ be a divisor.
    Let $f \in R(\fd, D)$ be a non-constant function and let $Z$ be the set
    on which $f$ attains its maximum. Then we say that
    $f$ \emph{splits} $D$ (at $x$) if
    $D(x) + \ord_x(f) > 0$ for some $x \in \boundary{Z}$.

    We will say $D$ \emph{does not split} (in $\fd$) if
    there is no $f \in R(\fd, D)$ that splits $D$.
\end{definition}

\begin{proposition}
    The set of divisors that do not split
    in $\fd$ is dense in $\fd$.
\end{proposition}

\begin{proof}
    Let $D \in \fd$.
    We will show that for all $\epsilon > 0$ there exists
    a $D' \in \fd$ that does not split and
    $d_\infty(D, D') < \epsilon$. We proceed by induction on
    $n = \deg D - \#\supp D$.
    If $n=0$, then $D$ has multiplicity $1$ at all points in its support, so
    there cannot be any function that splits $D$.

    Now, suppose $n > 0$ and
    suppose there exists an $f \in R(\fd, D)$ that splits
    $D$, then let $Z$ be the set on which $f$ attains its maximum and let
    \begin{equation*}
        f_\epsilon = f \oplus (\sup(f) - \epsilon).
    \end{equation*}
    Then $f_\epsilon$ is a weighted chip-firing move.
    For $\epsilon > 0$ small enough, $Z_\epsilon \cap \supp D = Z \cap \supp D$,
    and $\interior{Z_\epsilon} \setminus Z$ consists only of open intervals,
    where we let
    \begin{equation*}
        Z_\epsilon = \{x \in \Gamma \mid \dist(x, Z) \leq \epsilon\}.
    \end{equation*}
    
    Since $f$ splits $D$, there exists some $x \in \boundary Z$ such that
    $D(x) + \ord_x(f) > 0$, in particular $D(x) + \ord_x(f_{\epsilon/2}) > 0$.
    Let $D' = D + \fdiv(f_{\epsilon/2})$.
    If $D'$ does not split,
    we're done since $d_\infty(D, D') = \epsilon/2 <
    \epsilon$. Otherwise, since
    $\# \supp D' \cap Z_\epsilon > \# \supp D \cap Z_\epsilon$
    and $D'|_{Z_\epsilon^c} = D|_{Z_\epsilon^c}$, we see that we can apply the
    induction hypothesis to get divisor $D''$ that does not split with
    $d_\infty(D', D'') < \epsilon/2$. The result the follows by the triangle
    inequality.
\end{proof}

\begin{proposition}
    If $D \in \fd$ does not split,
    then the dimension of the tangent space of $\fd$ at
    $D$ is at least $r(\fd)$.
\end{proposition}

\begin{proof}
    We will show by induction on $r$ that if
    $r(\fd) \geq r$, then $\dim T_D\fd \geq r(\fd)$,
    which will imply the result.
    When $r = 0$ there is nothing to show, so suppose $r > 0$.

    Since $r(\fd) \geq 1$, there exists some non-constant
    $f \in R(\fd, D)$. Define
    \begin{equation*}
        f_t = f \oplus (\sup(f) - t).
    \end{equation*}
    Then $t \in [0, \epsilon) \mapsto D + \fdiv(f_t)$
    determines a tangent vector in $\fd$.

    Let $Z$ be the set on which $f$ attains its maximum
    and choose some $x \in \boundary Z$. Clearly, $D(x) \geq 1$,
    so $D \in \fd(-x)$.
    We claim that
    $R(\fd(-x), D)$ contains no function $g$ that
    attains its maximum at $x$ and $\ord_x(g) < 0$. Indeed, if $g$ was such a
    function and we denote by $Y$ the set on which $g$ attains its maximum,
    then $x \in Y$ and $D(x) + \ord_x(g) \geq 1$ by definition of $\fd(-x)$,
    which contradicts the fact that $D$ does not split.

    Now, it follows from the definition of rank that
    \[r(\fd(-x)) \geq r(\fd) - 1 \geq r - 1,\] 
    and hence 
    $\dim_D \fd(-x) \geq r - 1$ by the induction hypothesis,
    so let $\zeta_1, \dots, \zeta_{r-1}$ be independent tangents in the
    tangent space of $\fd(-x)$ at $D$. Since
    $\fd(-x) \subseteq \fd$, we may see $\zeta_i$ as tangents in the
    tangent space of $\fd$ at $D$.
    Since there is no function $g \in R(\fd(-x), D)$
    such that $\ord_x(g) \neq 0$ attains its maximum at $x$, we deduce
    that the $\zeta_i$ all correspond to infinitesimal transformations of $D$
    that fix the chips at $x$. However, since $x$ was chosen in the boundary of
    $Z$, $\zeta$ has a non-zero component among the coordinates
    corresponding to the chips of $D$ at $x$, so $\zeta$ is independent
    from the $\zeta_i$, which completes the proof.
\end{proof}

\begin{corollary}
    \label{cor:complex-dim}
    If $\fd$ admits the structure of an abstract polyhedral complex, then
    all its maximal faces have dimension at least $r(\fd)$.
\end{corollary}

\begin{proof}
    For any maximal face $\sigma$ of $\fd$, there exists
    some $D \in \relint(\sigma)$ which does not split (since $\relint(\sigma)$
    is open in $\fd$) and so 
    \[\dim(\fd) = \dim T_D(\fd) = \dim T_D\fd \geq
    r(\fd).\]
\end{proof}

\begin{corollary}
    \label{cor:fin-gen-dim}
    If $\fd$ is finitely generated, then $\fd$ is an abstract polyhedral
    complex whose maximal faces are all of dimension at least $r(\fd)$.
\end{corollary}

\section{The realizability problem}
\label{sec:realizability}

\subsection{Tropicalization}
\label{sec:tropicalization}

We will now describe the tropicalization process, which attributes to a
smooth projective curve a metric graph. We
will first recall a few definitions on algebraic curves and models.

\begin{definition}
An algebraic curve $C$ over $k$
is called \emph{pre-stable} if it is reduced and has only
ordinary double points as singularities. It is called \emph{stable} if in
addition 
\begin{enumerate}
    \item $C$ is connected and projective, of arithmetic genus $p_a(C) \geq 2$.
    \item If $Y$ is an irreducible component of $C$, which is isomorphic to
        $\proj^1_k$, then $Y$ meets the other components of $C$ in at least
        three points.
\end{enumerate}
If in the above we replace the requirement for three intersection points
with only
two intersection points, we obtain the definition of a \emph{semi-stable} curve.
The curve is called
\emph{totally degenerate} if all of its irreducible components are isomorphic to
the projective line over $k$ and all singularities of $C$ are $k$-rational.
\end{definition}

We may naturally associate to a pre-stable curve a weighted graph called
its \emph{dual graph}, the
vertices of which correspond to the irreducible components and the edges to the
nodes. The weights of the vertices
are given by the genus of the respective components. Figure \ref{fig:trop-ex}
depicts an example of a stable curve (in the center) and its corresponding
dual graph (on the right).

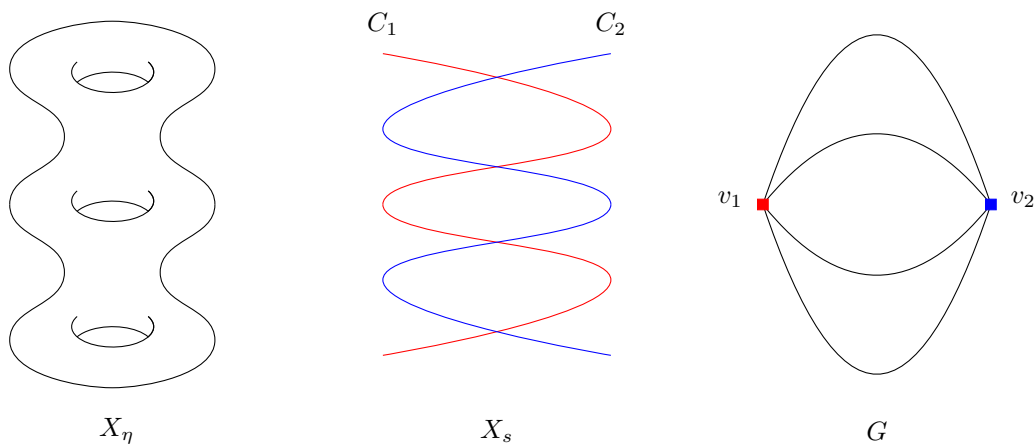
\begin{figure}[ht]
    \centering
    \begin{tikzpicture}
        \begin{scope}[scale=0.9, shift={(0.1, 0)}]
            \draw (0,0) .. controls (0,0.5) and (0.8,0.5) .. (0.8,1)
                        .. controls (0.8,1.5) and (0,1.5) .. (0,2)
                        .. controls (0,2.5) and (1,2.7) .. (1.5,2.7)
                        .. controls (2,2.7) and (3,2.5) .. (3,2)
                        .. controls (3,1.5) and (2.2,1.5) .. (2.2,1)
                        .. controls (2.2,0.5) and (3,0.5) .. (3,0);
            \draw (0,0) .. controls (0,-0.5) and (0.8,-0.5) .. (0.8,-1)
                        .. controls (0.8,-1.5) and (0,-1.5) .. (0,-2)
                        .. controls (0,-2.5) and (1,-2.7) .. (1.5,-2.7)
                        .. controls (2,-2.7) and (3,-2.5) .. (3,-2)
                        .. controls (3,-1.5) and (2.2,-1.5) .. (2.2,-1)
                        .. controls (2.2,-0.5) and (3,-0.5) .. (3,-0);
            \draw (2.025,0.2) arc (30:-210:0.6 and 0.3);
            \draw (2.025,0.2) arc (30:-30:0.6 and 0.3)
                arc (30:150:0.6 and 0.3);
            \draw (2.025,2.1) arc (30:-210:0.6 and 0.3);
            \draw (2.025,2.1) arc (30:-30:0.6 and 0.3)
                arc (30:150:0.6 and 0.3);
            \draw (2.025,-1.65) arc (30:-210:0.6 and 0.3);
            \draw (2.025,-1.65) arc (30:-30:0.6 and 0.3)
                arc (30:150:0.6 and 0.3);
        \end{scope}
        
        \begin{scope}[shift={(5,0)}]
            \draw[red] (0,0) .. controls (0,0.5) and (3, 0.5) .. (3, 1)
                .. controls (3,1.5) and (0, 2) .. (0, 2);
            \draw[red] (0,0) .. controls (0,-0.5) and (3, -0.5) .. (3, -1)
                        .. controls (3,-1.5) and (0, -2) .. (0, -2);
            \draw[blue] (3,0) .. controls (3,0.5) and (0, 0.5) .. (0, 1)
                        .. controls (0,1.5) and (3, 2) .. (3, 2);
            \draw[blue] (3,0) .. controls (3,-0.5) and (0, -0.5) .. (0, -1)
                        .. controls (0,-1.5) and (3, -2) .. (3, -2);
            \node at (0,2.4) {$C_1$};
            \node at (3,2.4) {$C_2$};
        \end{scope}

        \begin{scope}[shift={(10, 0)}]
            \draw (0,0) .. controls (1, 1.25) and (2, 1.25) .. (3, 0);
            \draw (0,0) .. controls (1, -1.25) and (2, -1.25) .. (3, 0);
            \draw (0,0) .. controls (1, 3) and (2, 3) .. (3, 0);
            \draw (0,0) .. controls (1, -3) and (2, -3) .. (3, 0);
            \node[vertex, red, scale=1.5, label={[left=4]$v_1$}] at (0,0) {};
            \node[vertex, blue, scale=1.5, label={[right=4]$v_2$}] at (3,0) {};
        \end{scope}
        
        \node at (1.5,-3.0) {$X_\eta$};
        \node at (6.5,-3.0) {$X_s$};
        \node at (11.5,-3.0) {$G$};

    \end{tikzpicture}
    \caption{Example of a smooth curve of genus 3 (left)
            degenerating to a stable curve (center) whose dual graph is a
            genus 3 banana graph (right). The four edges of the dual graph
        correspond to the four intersection points between the two irreducible
        components of the stable curve.}
    \label{fig:trop-ex}
\end{figure}

We may give this graph further structure if we consider the pre-stable curve
inside a fixed smoothing.
Let $K$ be a valued field with ring of integers $R = \cO_K$, and corresponding
maximal ideal $\fm_{\cO_K} = \fm$. Denote $k = R/\fm$ the residue field.
For simplicity, we will assume the residue field $k$ is algebraically closed and
that $R$ is complete.

\begin{definition}
    A \emph{fibered surface} over $S = \Spec{R}$ is an integral, projective,
    flat scheme $\pi: X \to S$ of dimension 1 over $S$. 
    Let $\eta$ be the generic point of $S$ and $s$ its only closed point. The
    fiber $X_\eta$ is called the \emph{generic fiber} and $X_s$
    the \emph{special fiber}.

    We say $X$ is a \emph{regular (resp. normal) fibered surface},
    whenever $X$ is a regular (resp. normal)
    scheme. We also call a regular fibered surface an \emph{arithmetic surface}.
    We will also say
    that $X$ is (pre/semi-)stable and/or totally degenerate whenever
    these properties hold for the special fiber $X_s$.
\end{definition}

When $X$ is pre-stable, we may equip the dual graph of $X_s$ with the structure
of a metric graph.

\begin{proposition}\cite[Corollary 10.3.22]{liu}
    \label{prop:node-local-eq}
    Let $X$ be a pre-stable fibered surface over $S$ such that $X_\eta$ is
    smooth.
    Let $x \in X_s$ be a singular point of $X_s$.
    Then we have an isomorphism
    \begin{equation*}
        \completion{\cO}_{X, x}
        \cong \completion{\cO}_{K}[[u, v]]/(uv - c)
    \end{equation*}
    for some $c \in \fm_{\cO_{K}}$.
\end{proposition}

\begin{definition}
    In the setting of Proposition \ref{prop:node-local-eq},
    let $w_x \geq 1$ be the valuation of $c$.
    We call $w_x$ the \emph{thickness} (or \emph{width}) of $x$ in $X$.
\end{definition}

\begin{definition}
    Let $X$ be a smooth, geometrically connected, projective curve over $K$.
    A normal
    fibered surface $\fX \to S$ such that 
    $\fX_\eta \cong X$ is a called \emph{model of $X$ over $S$}.
\end{definition}

\begin{remark}
    There may be many different models of any given curve $X$, but they might
    not be pre-stable. As we will soon see, we will still be able to uniquely
    attribute a metric graph
    to $X$, which will be what we call 
    the \emph{tropicalization} of $X$.
\end{remark}

\begin{theorem}\cite[Corollary 2.7]{stable-reduction}
    \label{stable-reduction}
    When $X$ is a smooth, projective, geometrically connected curve over $K$,
    with
    $g(X) \geq 2$, there exists
    a finite algebraic extension $L$ of $K$,
    such that $X_L = X \times_K L$ has a unique stable model
    $\fX_L$ over $\cO_L$ with generic fiber isomorphic
    to $X_L$. Moreover, $L$ can be taken separable over $K$.
\end{theorem}

\begin{remark}
    The theorem stated in this form can be found in \cite[Theorem 10.4.3]{liu}.
\end{remark}

\begin{definition}[Tropicalization]
    Let $X$ be a smooth, projective, geometrically connected curve over $K$. Let
    $L$ be a finite algebraic extension of $K$ such that $X_L$ has a unique
    stable model $\fX_L$ as in Theorem \ref{stable-reduction}.

    We equip the dual graph $G$ of $\fX_s$ with the
    structure of a metric graph by defining the length function
    \begin{align*}
        l: E(G) &\to \Q\\
        e &\mapsto \frac{w_{n_e}}{[L : K]},
    \end{align*}
    where $n_e$ denotes the node
    of $\fX_s$ corresponding to the edge $e$. The resulting metric graph
    $\Gamma$ is called the \emph{tropicalization} of $X$.
\end{definition}

\begin{remark}
    The fact that the above definition does not depend on choice of 
    the field extension
    $L\mid K$ is verified in 
    \cite[Lemma 2.2.4]{viviani}.
\end{remark}

\subsection{Specialization of divisors}
\label{sec:specialization}

A first natural object we might want to transfer from the algebraic curve to the
tropicalization are divisors. We will now describe the specialization process as
introduced in \cite{specialization-lemma}.

Let $X$ be a smooth, geometrically connected, projective curve of genus
$\geq 2$. Up to performing base change to a finite extension of $X$, suppose
$X$ admits a stable model $\fX$. So let $\Gamma$ be the tropicalization of
$X$.
By taking the minimal desingularization of
$\fX$ we obtain the \emph{minimal regular model} of $X$, denoted
$\fX_{\min}$. 
\cite[Corollary 10.3.25]{liu} tells us that this model is semi-stable and
that the dual graph corresponds to the subdivision of the edges of $\Gamma$,
so that each edge has length 1.

Denote by $C$ the special fiber of $\fX_{\min}$ and let
$C_1, \dots, C_n$ be its irreducible components,
corresponding to the vertices $v_1, \dots, v_n$ of the dual graph.
For a $K$-rational point $P$ of $X$,
we may take its Zariski closure in $\fX_{\min}$
to obtain a Weil divisor on $\fX_{\min}$
which we denote by $\closure{P}$. We have that
$\closure{P}$ intersects the special fiber $C$
in its smooth locus in a single point by \cite[Proposition 9.1.32]{liu} 
and so $\closure{P}$ intersects a unique irreducible
component of $C$. Let $v(P)$ be the corresponding
vertex in $\Gamma$.

This allows us to define a map $\rho: \Div(X(K)) \to \Div(\Gamma)$ by setting
\begin{equation*}
    \rho(D) = \sum_{P \in X(K)} D(P) \cdot v(P).
\end{equation*}
We would like to
extend the definition of $\rho$ to all of $X(\closure{K})$. In order to do so,
we need to check that $\rho$ is compatible with base change. Let $L$ be a finite
extension of $K$, and
denote
$\rho_K: \Div(X(K)) \to \Div(\Gamma)$ and $\rho_L: \Div(X(L)) \to \Div(\Gamma)$
the corresponding maps. We 
want to show that $\rho_L|_{\Div(X(K))} = \rho_K$. 
Taking the base change $\fX_{\min, K} \times_{\cO_K} \cO_L$
yields a semi-stable model of
$X_L$. By \cite[Corollary 10.3.22(c)]{liu},
this has the effect of multiplying the
thickness of the nodes of the special fiber by $[L:K]$.
This introduces new singularities to the
model, and the minimal regular model of $X_L$ is given
by repeatedly blowing up those singularities. The
effect of this on the dual graph is to subdivide the edges into $[L:K]$ segments
of equal length. Note moreover that if $P$ is a $K$-rational point, it will
specialize to the same connected component after performing base-change and
desingularizing. So in fact the two specialization maps agree on divisors that
are supported on $K$-rational points. See \cite[Section 2]{capacity-pairing} for
more details on the compatibility with base change.

Since the specialization map is compatible with base change, it induces a map
\[\rho: \Div(X(\closure{K})) \to \Div(\Gamma)\]
that has image in $\Div_\Q(\Gamma)$, which is the set of divisors supported on 
the points in $\Gamma$ which
have rational distance from any given vertex.

We now define the rank of a divisor on an algebraic curve in analogy to
Definition \ref{def:rank}.
\begin{definition}
    If $D$ is a divisor on $X$, we define
    \begin{equation*}
        r(D) := 
            \max \{d \in \N \mid |D - E| \neq \emptyset\textrm{
            for all effective divisor $E$ of degree $d$}\},
    \end{equation*}
    to be the \emph{rank} of $D$.
\end{definition}
\begin{remark}
    By \cite[Lemma 2.4]{specialization-lemma} the rank of a divisor
    $D$ is equal to $\dim\mathcal{L}(D) - 1$. This explains the formula of the
    Riemann-Roch theorem (Theorem \ref{thm:trop-rr}) in terms of the rank.
\end{remark}

Matt Baker has famously shown that during the specialization process,
the rank of a divisor
can only increase.
\begin{lemma}[Specialization lemma]\cite[Lemma 2.8]{specialization-lemma}
    \label{lemma:specialization}
    For all divisors 
    $D \in \Div(X(K))$,
    \begin{equation*}
        r(\rho(D)) \geq r(D).
    \end{equation*}
\end{lemma}

\begin{example}
    Consider the situation in Figure \ref{fig:trop-ex}. Let $D = v_1 + v_2$.
    Then $D$ is a divisor of rank $1$. We may lift $D$ to a divisor $D'$ 
    of $X_\eta$, since by \cite[Remark 2.3]{specialization-lemma} the
    specialization map is surjective.
    When $X_\eta$ is not hyperelliptic, the divisor $D'$ must be of rank $0$,
    which gives an example of a case where the inequality in Lemma
    \ref{lemma:specialization} is strict.

    We will now construct a model as in Figure \ref{fig:trop-ex}. 
    Let $K$ be a valued field with ring of integers $R$, maximal ideal $\fm$ and
    residue field $k = R/\fm$.
    Let $q_1, q_2$ be two homogeneous polynomials over $R$ of
    degree 2 such that
    $X_s = V(\overline{q_1}\overline{q_2}) \subseteq \proj^2_k$
    is a union of two smooth
    quadrics in general position. By Bézout's theorem, these quadrics will
    intersect in four points, hence the dual graph of $X_s$ is $G$ as in Figure
    \ref{fig:trop-ex}.
    Let $c \in \fm$
    and let $p$ be a homogeneous
    degree 4 polynomial
    such that $X_\eta = V(q_1q_2 + cp) \subseteq \proj^2_K$ is a smooth quartic.
    It is then clear that $X = V(q_1q_2 + cp) \subseteq \proj^2_R$ is a fibered
    surface with generic fiber $X_\eta$ and special fiber $X_s$. By the
    degree-genus forumla $X_\eta$ is of genus 3.
    To conclude, note that by
    \cite[Example IV.5.2.1]{hartshorne} any quartic plane curve is
    non-hyperelliptic.
\end{example}

When $\Gamma$ is a metric graph, we call a divisor $D \in \Div(\Gamma)$
\emph{realizable}, whenever there exists a curve $X$ and $D'$ an effective
divisor on $X$ of the same rank as $D$ such that $(\Gamma, D)$ is the
tropicalization of $(X, D')$.
It is an important open problem to characterize realizable divisors.

\subsection{Realizability of canonical divisors}
\label{sec:realizability-canonical}

In \cite{realizability-canonical}, the authors give 
a complete characterization for the realizability of divisors in the
canonical tropical linear system.
We will reinterpret this result in a simpler context and give some
sufficient conditions for realizability.

The condition as presented in \cite{realizability-canonical} works with more
structure on the graph $\Gamma$. First of all, the vertices are decorated with a
function $h: V \to \N$. During the tropicalization process, $h$ records the
genus of the corresponding irreducible component. With this decoration, the
canonical divisor on $\Gamma$ is defined as
\begin{equation*}
    K = \sum_{x \in V}(2h(x) - 2 + \val(x)) \cdot x
\end{equation*}
and the genus is
\begin{equation*}
    g = b_1(\Gamma) + \sum_{x \in V}h(x),
\end{equation*}
where $b_1(\Gamma)$ is the first betti number of $\Gamma$. Note that when the
semi-stable reduction is totally degenerate, these definitions agree
with the previous ones.

\begin{definition}
    A \emph{graph with legs} is a length space obtained from a metric graph by
    attaching to it a finite set of half-rays, which we call \emph{legs}.
    The notions from Section \ref{sec:metric-graphs} extend naturally
    to graphs with legs.

    Let $\Gamma$ be a graph with legs.
    We call any function $l: V \to \Z_{\leq 0}$ such that $l^{-1}(0) \neq
    \emptyset$ a \emph{level function} on $\Gamma$. Such a level function
    induces a full order on the vertices of $\Gamma$.
    We call $\Gamma$ with the data of a level function a \emph{level graph},
    and denote it by $\overline{\Gamma}$.
    For any edge $e$ between two vertices $x, y$, we say $e$ is
    \emph{horizontal} whenever $l(x) = l(y)$, otherwise we say $e$ is vertical.

    We write $\Lambda = \bigsqcup_{x \in V}T_x\Gamma$ for the
    set of tangent vectors of $\Gamma$ based at the vertices.
    This is naturally identified with the set of half-edges and legs.

    An \emph{enhanced level graph} $\Gamma^+$
    is a level graph $\overline{\Gamma}$
    together with a function $k: \Lambda \to \Z$ such that
    \begin{enumerate}
        \item For any edge $e$ with corresponding tangents $\zeta^+$,
            $\zeta^-$ along $e$,
            we have that $k(\zeta^+) + k(\zeta^-) = -2$.
            An edge is horizontal iff $k(\zeta^\pm) = -1$ and when $e$ is
            vertical with $\zeta^+$ being the tangent at the higher
            vertex, then $k(\zeta^+) > k(\zeta^-)$.
        \item For each vertex $v$,
            \begin{equation*}
                \sum_{\zeta \in T_v\Gamma} k(\zeta) = 2h(v) - 2.
            \end{equation*}
    \end{enumerate}

    When $\Gamma^+$ is an enhanced level graph,
    we define the \emph{type} $\mu(v)$ of a vertex to be the ordered tuple (in
    decreasing order) of the 
    $k(\zeta)$, for $\zeta \in T_v(\Gamma)$.
\end{definition}

\begin{definition}
    \label{def:inconvenient-vertex}
    Let $\Gamma^+$ be an enhanced level graph.
    A vertex $v \in \Gamma^+$ is called \emph{inconvenient}
    if $h(v) = 0$ and its type $\mu(v) = (k_1,
    \dots, k_n)$ has the following properties:
    \begin{itemize}
        \item $k_i \neq -1$ for all $i$.
        \item There exists an index $i$ such that
            \begin{equation*}
                k_i > \left(\sum_{k_j < 0}-k_j\right) - \#\{k_j < 0\} - 1
            \end{equation*}
    \end{itemize}
\end{definition}

We can naturally attribute to a metric graph
$\Gamma$ along with a given canonical divisor
$D \in |K|$ an enhanced level graph.

\begin{definition}
    \label{def:enhanced-level-graph}
    Let $\Gamma$ a metric graph and let
    $D = K + \fdiv(f) \in |K|$ an effective canonical divisor on $\Gamma$.
    Up to subdividing the model of $\Gamma$, we may
    assume that $D$ is supported on the vertices of $\Gamma$.
    \begin{itemize}[itemsep=0em]
        \item
            For each vertex $x$,
            we attach $D(x)$ legs to $\Gamma$ and call the resulting graph with
            legs $\Gamma'$.
        \item 
            Extend $f$ to a rational function on $\Gamma'$ so that $f$ is linear
            on the legs with $s_\zeta(f) = -2$ for $\zeta$ a tangent at a vertex
            along a leg.
        \item 
            Give $\Gamma'$ the structure of level map
            induced by the function $f$.
        \item
            Define also $k(\zeta) = -s_\zeta(f) - 1$.
    \end{itemize}
    This equips $\Gamma'$ the structure of an enhanced
    level graph, which we denote
    by $\Gamma^+(f)$.
\end{definition}

\begin{definition}
    We will say a vertex $v$ of $\Gamma$ is inconvenient if it is an
    inconvenient vertex of the enhanced level graph $\Gamma^+(f)$.
\end{definition}

We may now state \cite[Theorem 6.3]{realizability-canonical}, which gives us a
necessary and sufficient condition for the realizability of a divisor on a
tropical curve. A \emph{cycle} is a non-trivial path from a vertex from itself.
A cycle is \emph{simple}
if it's not self-intersecting (other than at the endpoints).

\begin{theorem}
    \label{thm:realizability-iff}
    Let $\Gamma$ be a metric graph and
    let $D = K + \fdiv(f)$ be an effective canonical divisor on $\Gamma$.
    We consider the model of $\Gamma$ subdivided so that $\bend f$ is
    supported on the vertices.
    Then $D$ is realizable if and only if the following conditions
    are satisfied:
    \begin{itemize}
        \item Every inconvenient vertex is contained in a simple cycle that lies
            above it (in the sense that $f(Z) \geq f(v)$, where $Z$ is the given
            cycle).
        \item Every horizontal edge (meaning that $f$ is constant on that
            edge) is contained in a simple cycle that lies
            above it.
    \end{itemize}
\end{theorem}

We will now describe more explicitly what it means for a vertex to be
inconvenient.

\begin{proposition}
    \label{prop:def-inconvenient}
    Let $\Gamma$ be a metric graph and
    $D = K + \fdiv(f)$ an effective canonical divisor in $|K|$.
    Let $v$ a vertex and denote $s_1, \dots, s_r$ the outgoing slopes of $f$
    along the edges adjacent to $v$.
    Then $v$ is inconvenient iff $h(v) = 0$,
    $s_j \neq 0$ for all $j$,
    and there is an $i$ such that $s_i < 0$
    and 
    \begin{equation*}
        -s_i > \sum_{j, s_j > 0} s_j.
    \end{equation*}
\end{proposition}

\begin{proof}

Let $v$ be a vertex,
since $k(\zeta) = -s_\zeta(f) - 1$
for tangents $\zeta \in \Lambda$,
the first condition in Definition \ref{def:inconvenient-vertex}
translates to $s(\zeta)
\neq 0$ for all tangents at $v$.

Let $\zeta_j$ be the tangents corresponding to the $k_j$ appearing 
in the type of the vertex $v$. Denote also $s_j = s_{\zeta_j}(f)$.
The second condition of Definition \ref{def:inconvenient-vertex} is equivalent
to the existence of an index
$i$ such that
\begin{equation*}
    -s_i > \sum_{s_j > 0} s_j.
\end{equation*}
From Definition \ref{def:enhanced-level-graph} it follows
that for all $j$ such that $s_j > 0$, $\zeta_j$ is a tangent along an
edge.
Now, if $\zeta_i$ is tangent along a leg,
the left hand side is just $2$. In this case we would get 
$2 > \sum_{s_j > 0} s_j$. For this to be satisfied there should be
at most one tangent with positive outgoing slope at $v$ and that slope
has to be equal to $1$.
But then there are at least $\val_\Gamma(x) - 1$
other edges with strictly negative outgoing
slopes, which would imply that $\fdiv(f)(x) \leq -\val_\Gamma(x) + 2$ (here the
divisor of $f$ is taken inside $\Gamma$).
This in turn means that
$(K + \fdiv(f))(x) \leq 0$,
which forces $(K + \fdiv(f))(x) = 0$, as $D = K + \fdiv(f)$ is by
assumption effective.
But then by definition of the enhanced level graph $\Gamma^+(f)$, we
would not have attached any legs, so this situation cannot happen.
We deduce that all the terms appearing in the inequality can only 
come from tangents along edges (so tangents of the original metric graph
$\Gamma$), which finishes the proof.
\end{proof}

Thanks to this, when talking about realizability we will not need to refer to
the structure of enhanced level graphs, and
so we will restrict our discussion to metric
graphs as defined in Section \ref{sec:metric-graphs}.

Denote by $\tropicalhodge$ the the moduli space
parametrizing isomorphism classes of
metric graphs (with vertex weights)
of genus $g$ with the choice of a canonical divisor.
It carries the structure of a generalized
cone complex by \cite[Theorem 4.3]{tropical-hodge}.
Let $\realizablelocus \subseteq \tropicalhodge$,
the subset of pairs $([\Gamma], D) \in \tropicalhodge$ that are
realizable. By \cite[Theorem 6.6]{realizability-canonical},
$\realizablelocus$ is an abstract cone complex
whose maximal cones have
dimension $4g-4$. 
Furthermore, by \cite[Proposition 6.9(i)]{realizability-canonical} the graphs
appearing in the maximal cones have $h \equiv 0$.
Hence the set of pairs $([\Gamma], D) \in \realizablelocus$, which have $h
\equiv 0$ is dense in $\realizablelocus$. For this reason in our
following discussion we
will only focus on the case where $h \equiv 0$ and so we
will be looking only at graphs that appear as the dual of a totally degenerate
semi-stable curve.

\subsection{Realizability locus of the canonical linear system}
\label{sec:canonical-linsys}

We will now give some characterizations of the realizability locus of the
canonical
complete linear system $|K|$, which is the subset of divisors which are
realizable.

\begin{proposition}
    \label{prop:realizable-convex}
    Let $\Gamma$ be a metric graph and let $\Real(|K|)
    \subseteq |K|$ be the set of
    realizable canonical divisors. Then $\Real(|K|)$ is tropically convex.
\end{proposition}

\begin{proof}
    Let $M$ be the set of $f \in R(K)$ that correspond to realizable
    divisors $K + \fdiv(f)$,
    plus the element $-\infty$. We claim that $M$ is a submodule of
    $R(K)$.

    Let $f \in M$.
    Since for all $c \in \R$, we have that $\fdiv(c\odot f) = \fdiv(f)$,
    we deduce that $K + \fdiv(c\odot f)$ and so $M$ is stable under tropical
    scalar multiplication.

    When $f,g \in M$, we will show that $f \oplus g = \max(f,g) \in M$.
    Consider the model $G$ of $\Gamma$, such that both $f$ and $g$
    are linear when restricted to any given edge.
    It follows that when restricted to any fixed edge,
    $f \oplus g$ is equal to either $f$ or $g$.

    If $e$ is a horizontal edge of $\Gamma$ with respect to $f \oplus g$,
    then we must have that $e$ is a horizontal edge with respect to one of
    $f$ or $g$. W.l.o.g. $f$ is constant on $e$, but then as $K + \fdiv(f)$
    is realizable, there is a simple cycle $\gamma \subseteq \Gamma$ containing
    $e$ and which lies above it with respect to $f$.
    But $f \oplus g \geq f$ and hence $\gamma$ also lies above $e$ with respect
    to $f \oplus g$. We deduce that the condition on the edges in Theorem
    \ref{thm:realizability-iff} is satisfied.

    Now, suppose $v$ is an inconvenient vertex with respect to $f \oplus g$.
    If $f(v) \neq g(v)$, then
    w.l.o.g. $f(v) > g(v)$ and so $f \oplus g$ agrees with $f$ in a
    neighbourhood of $v$. This implies that $v$ is inconvenient with respect
    to $f$ as well. Hence there is a simple cycle $\gamma \subseteq \Gamma$
    containing $v$ and which lies above it with respect to $f$, and as before
    this also implies it lies above $v$ with respect to $f \oplus g$.
    Now, suppose $f(v) = g(v)$. In this case,
    for any tangent $\zeta \in T_v\Gamma$, we have
    that $s_\zeta(f\oplus g) = \max(s_\zeta(f), s_\zeta(g))$.
    Fix any ordering $\zeta_1, \dots, \zeta_r$ on $T_v\Gamma$ and let
    $s_i = s_{\zeta_i}(f \oplus g)$ and $s_i' = s_{\zeta_i}(f)$.
    It follows that $s_i \geq s_i'$ for all $i$.
    Since $v$ is inconvenient with respect to $f \oplus g$, we know that there
    exists some $i$ such that $s_i < 0$ and
    \begin{equation*}
        -s_i > \sum_{j, s_j > 0}s_j
    \end{equation*}
    Then we have that
    \begin{equation*}
        -s_i' \geq -s_i > \sum_{j, s_j > 0}s_j
        \geq \sum_{j, s_j' > 0}s_j \geq \sum_{j, s_j' > 0}s_j'.
    \end{equation*} 
    Hence it follows that $v$ is also an inconvenient vertex for $f$, 
    and as before this yields a simple cycle that lies above $v$ with respect to
    both $f$ and $f \oplus g$.
    We deduce that the condition on inconvenient vertices in Theorem
    \ref{thm:realizability-iff} is satisfied and so
    $K + f\oplus g$ is realizable.

    We conclude that $M$ is a submodule of $R(K)$ and hence $\Real(|K|)$
    being the
    image of this submodule is tropically convex.
\end{proof}

\begin{proposition}
    \label{prop:rl-definable}
    The realizability locus $\Real(|K|)$
    is a definable subset of $|K|$.
\end{proposition}

\begin{proof}
    We know that $R(K)$ is finitely generated and so let $\{\phi_1, \dots,
    \phi_r\}$ be a generating set. Consider a model $G = (V, E)$ of $\Gamma$,
    such that the $\phi_i$ are all linear on each edge. This is possible since
    the set
    \[\bigcup_{i = 1}^r \bend(\phi_i)\]
    is finite, so we may choose $V$ to contain this set.
    The set of $a_i$ such that $\max(a_i + \phi_i) \geq 0$ on a fixed edge is
    clearly a definable subset of $\R^r$.
    Denote this subset by $C_e$.
    Let also $\gamma$ be a simple path in $\Gamma$, then the set of $a_i$
    such that $\max(a_i + \phi_i) \geq 0$ an all of $\gamma$ is just the
    intersection of all the $C_e$ for each $e \subseteq \gamma$.
    Hence this set is also definable and we will denote it by
    $C_\gamma$.

    Now, choose a vertex $v$. Let $S$ be any subset of $\{1, \dots, r\}$ such
    that the rational function
    \[\phi_S := \bigoplus_{i \in S} (\phi_i - \phi_i(v))\]
    makes $v$ an inconvenient vertex.
    Then for any choice of $a_1, \dots, a_r$ such that
    $a_i = -\phi_i(v)$ for all $i \in S$ and $a_i < -\phi_i(v)$, we have that
    \[\phi := \bigoplus_{i = 1}^r a_i\odot \phi_i\]
    also makes $v$ inconvenient. Indeed, we have that $\phi$ and $\phi_S$
    coincide in a neighbourhood of $v$,
    and hence the outgoing slopes of $\phi$ at $v$ are all equal to the outgoing
    slopes of $\phi_S$.
    Denote the set of such $(a_1, \dots, a_r) \in \R^r$ by
    $I_{v, S}$. Again, $I_{v, S}$ is clearly a definable set.
    Define now
    \[I_v := \bigcup_{S \subseteq \{1, \dots, r\}}I_{v, S},\]
    where the union ranges over the subsets $S$ such that $v$ is inconvenient
    for $\phi_S$.
    This set corresponds to all the rational functions $f \in R(K)$ such that
    $f(v) = 0$ and for which $v$ is inconvenient.
    Indeed, if $f$ is such a function, then
    \[f = \bigoplus_{i = 1}^r a_i\odot \phi_i\]
    for some choice of $a_i$ and since $f(v) = 0$, we necessarily have that
    $a_i \leq -\phi_i(v)$ for all $i$ and $a_i = -\phi_i(v)$ at least for one
    $i$.
    We may then set $S = \{i \mid a_i = -\phi_i(v)\}$ for which $f \in I_{v,S}$.

    Now, the set
    \[Z_v := I_v \cap \bigcup_\gamma C_\gamma,\]
    where the union is over the simple cycles of $\Gamma$ containing $v$
    corresponds to all the rational functions $f \in R(K)$ such that
    $f(v) = 0$, $v$ is inconvenient and contained in some simple cycle that lies
    above it.

    The image of $Z_v$
    via the piece-wise affine map
    \begin{align*}
        \eta: \R^r &\to |K|\\
        (a_1, \dots, a_r) &\mapsto K
        + \fdiv\left(\bigoplus_{i = 1}^r a_i \odot \phi_i\right)
    \end{align*}
    is precisely the subset of divisors of $|K|$ such that $v$ is inconvenient,
    but contained in a simple cycle that lies above it.
    Similarly, we have that $\eta(I_v)$ is the set of divisors of $|K|$ such
    that $v$ is inconvenient (with no further conditions).
    Note also that since $\eta$ is piece-wise affine, the image of any definable
    set is again
    definable.
    
    Let $U$ be a subset of vertices. We let 
    \[A_U := \bigcap_{v \in U} \eta(Z_v) \cap \bigcap_{v \notin U} \eta(I_v)^c.\]
    This is the set of divisors in $|K|$ for which the set of inconvenient
    vertices is precisely $U$ and every inconvenient vertex is
    contained in a simple cycle that lies above it.

    It follows that $A := \bigcup_{U \subseteq V}A_U$ is the set of divisors in
    $|K|$ for which each inconvenient vertex is contained in a simple cycle that
    lies above it. It is clear from the construction of $A$ that this set is
    definable.

    We will now show using an analogous argument that the set of divisors in 
    $|K|$ for which each horizontal edge is contained in a simple cycle lying
    above it is also definable, which will finish the proof.

    Let $e$ be an edge and let $\phi_j$ be a rational function that is constant
    on $e$. Let $H_{e, j}$ be the set of $(a_1, \dots, a_r) \in \R^r$ 
    with $a_j = -\phi_j(e)$ and
    \[\min_{x \in e} \max_{i \neq j} (a_i + \phi_i) < a_j + \phi_j(e).\]
    Then $H_{e, j}$ is clearly definable and corresponds to the set of rational
    functions that have a horizontal segment along $e$, on which they are equal
    to $a_j + \phi_j(e) = 0$.
    Define also $H_e := \bigcup_{j=1}^r H_{e,j}$, where we just let $H_{e, j} =
    \emptyset$ when $\phi_j$ is not constant on $e$. Then $\eta(H_e)$
    is the set of
    divisors in $|K|$ that have a horizontal segment on the edge $e$.
    Let 
    \[Y_e := H_e \cap \bigcup_\gamma C_\gamma,\]
    where the union is over the simple cycles of $\Gamma$ containing $e$.
    Then clearly $\eta(Y_e)$ is the set of divisors in $|K|$ such that $e$ is
    has a horizontal segment, which is contained in a cycle that lives above it.
    Finally, if we let
    \[B = \bigcup_{F \subseteq E}
    \bigcap_{e \in F} \eta(Y_e) \cap \bigcap_{e \notin F} \eta(H_e)^c,\]
    we deduce that $B$ is the set of divisors in $|K|$ for which each horizontal
    edge is contained in a simple cycle that lies above it.

    By construction $B$ is definable, and since $\Real(|K|) = A \cap B$, the
    realizability locus is definable as well.
\end{proof}

\begin{proposition}
    \label{prop:rl-closed}
    The realizability locus $\Real(|K|)$ is closed.
\end{proposition}

\begin{proof}
    Let $D_n = K + \fdiv(f_n)$ be a sequence of realizable divisors in 
    $\Real(|K|)$ converging
    in $|K|$ to $D + \fdiv(f)$. Suppose $v$ is an inconvenient vertex for
    $f$. Let $\zeta$ be a tangent vector at $v$. Let $s$ be the slope of $f$
    along $\zeta$ and $s_n$ the slope of $f_n$ along $\zeta$.
    The sequence $s_n$ takes values in the finite set
    $\{-\deg{K},\dots, \deg{K}\}$, so up to switching to a subsequence, we may
    assume that the $s_n$ are all equal to some $s'$.
    Let $e$ be the edge corresponding to $\zeta$, then since 
    $K + \fdiv(f_n)$ is effective, all of the $f_n$ are convex along the edge
    $e$. This forces $s \geq s'$. Hence up to switching to a subsequence, we may
    assume that for all tangent vectors $\zeta$ at $v$, we have that
    \begin{equation*}
        s_\zeta(f) \geq s_\zeta(f_n).
    \end{equation*}
    We deduce the same way as in the proof of Proposition
    \ref{prop:realizable-convex} that $v$ is also an inconvenient vertex for all
    of the $f_n$. Since all the $D_n$ are realizable, we deduce that
    for all $n$ there exists some simple cycle $\gamma_n$
    that lie above $v$ (with respect to $f_n$).
    There are only finitely many simple cycles, so up to switching to a
    subsequence, we may assume that all of the $\gamma_n$ are equal to some
    fixed cycle $\gamma$. We have for all $n$ that $f_n(\gamma) \geq f_n(v)$ and
    so by taking the limit, we also obtain that
    $f(\gamma) \geq f(v)$. Hence we conclude that for every inconvenient vertex
    for $f$, there exists a simple cycle that lies above it.

    Now, let $e$ be a horizontal edge for $f$. We have that
    \[\lim_{n \to \infty}\|f - f_n\|_\infty = 0,\]
    and the $f_n$ are convex on $e$ and have only integral slopes. For 
    $n$ such that $\|f - f_n\| \leq l(e) < 2$ this implies $f_n$ has to have a
    horizontal section along $e$. We deduce that there exists a simple 
    cycle $\gamma_n$ containing $e$ such that $f_n(\gamma_n) \geq f_n(e)$. Like
    before, we deduce that there exists a simple cycle $\gamma$ containing 
    $e$, such that $f(\gamma) \geq f(e)$.

    The two conditions for realizability from Theorem \ref{thm:realizability-iff}
    are satisfied, and so we conclude that $D$ is realizable, and so
    $\Real(|K|)$ is closed.
\end{proof}

\begin{corollary}
    The realizability locus $\Real(|K|)$ admits the structure of an abstract
    polyhedral complex.
\end{corollary}

\begin{proof}
    By Propositions \ref{prop:rl-definable} and \ref{prop:rl-closed},
    $\Real(|K|)$ is closed and definable, so the statement follows from 
    Corollary \ref{cor:closed-def-complex}.
\end{proof}

Although we have shown that $\Real(|K|)$ is an abstract polyhedral complex, this
does not show anything about whether or not it is finitely generated. For
example, consider the triangle spanned by
$[1: 0: 0]$, $[0: 1: 0]$ and $[1: 1: 0]$
in $\T(\R^3)$. It is impossible to express
any point of the form $[a: b: 1]$ with $a + b = 1$ as a convex combination of 
other points in the triangle, so the triangle is not finitely generated, despite
being a tropically convex subspace.

If one considers the tropically convex subspace of $|K|$ spanned by the
realizable extremals of $|K|$, it would appear that this set agrees with
$\Real(|K|)$ when one consider genus 3 graphs. Unfortunately, this already fails
for genus 4 graphs. For example, the divisor shown in Figure
\ref{fig:canonical-span} is realizable, but
does not belong to the span of realizable extremals of $|K|$ (this is not
obvious
a priori, but has been checked using a computer program).

\begin{figure}
    \centering
    \begin{tikzpicture}[scale=1.75]
        \coordinate (a) at (0,0);
        \coordinate (b) at (1,-1);
        \coordinate (c) at (1,0);
        \coordinate (d) at (1,1);
        \coordinate (e) at (2,0);
        \coordinate (f) at (3,0);
        \draw (a) -- (b) -- (c) -- (d) -- (a) -- (c);
        \draw (b) -- (e) -- (d);
        \draw (e) -- (f);
        \draw ($(f) + (0.5, 0)$) circle (0.5);
        \begin{scope}[every node/.style=vertex]
            \node at (0, 0) {};
            \node at (1, -1) {};
            \node at (1, 0) {};
            \node at (1, 1) {};
            \node at (2, 0) {};
            \node at (3, 0) {};
        \end{scope}
        \node[divisor, label=1] at (d) {};
        \node[divisor, label=1] at (e) {};
        \node[divisor, label=4] at ($(d)!0.5!(e)$) {};
    \end{tikzpicture}
    \caption{Realizable divisor in the canonical linear system that is not in the
    span of the realizable extremals of $|K|$.}
    \label{fig:canonical-span}
\end{figure}
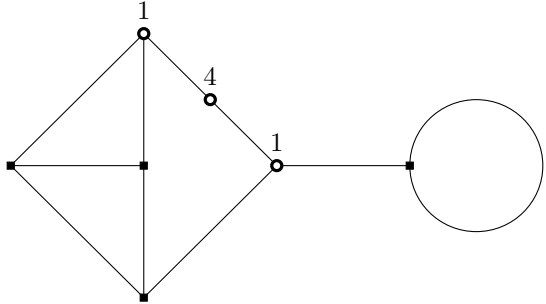

\begin{question}
    Is the realizability locus in the canonical linear system finitely
    generated?
\end{question}

An affirmative answer to this question would be very useful,
as it would allow to characterize the realizability locus using its extremals.

\subsection{Cycles and realizability}
\label{sec:cycles}

In this subsection we will give a sufficient characterization for realizability
of divisors in the canonical linear system
by studying the cycles that can appear in the metric graph.
We will be able to deduce that $\Real(|K|)$ always contains some maximal cell of
$|K|$ of dimension $g-1$.

\begin{proposition}
    \label{prop:local-max}
    Let $D = K + \fdiv(f)$ an effective canonical divisor.
    If $C$ is a local maximum of $f$,
    then $C$ does not have vertices of valence $1$ in $C$.
    In particular, $C$ contains
    a cycle.
\end{proposition}

\begin{proof}
    For any vertex $v \in C$, since $f$ may not have positive slope
    along any tangent at $x$ (as $C$ is a local maximum), we deduce that
    \begin{equation*}
        D(v) \leq K(v) - (\val_\Gamma(v) - \val_C(v)) = \val_C(v) - 2.
    \end{equation*}
    In particular, $\val_C(v) \geq 2 + D(v) \geq 2$ and so $v$ is not a leaf
    in $C$.

    The second statement follows from the fact that if $G$ is a connected 
    graph with edges $E$, vertices $V$, each of valence at least 2,
    then $\#E = \frac{1}{2}\sum_{v \in V} \val_G(v) \geq \#V$ and so
    $g(G) = \#E - \#V + 1 \geq 1$.
\end{proof}

\begin{lemma}
    \label{lemma:inconvenient-slopes}
    If $v \in V$ is an inconvenient vertex, there are at least two edges with
    strictly positive outgoing slope.
\end{lemma}

\begin{proof}
    Suppose $v$ is of valence $r$ and denote the outgoing
    slopes by $s_1, \dots, s_r$, in increasing order. If $s_r$ is the only
    strictly positive number, we have by definition of inconvenient vertex
    (Proposition \ref{prop:def-inconvenient}) that
    $-s_1 > s_r$ and $s_i < 0$ for all $i < r$.
    But then 
    \begin{equation*}
        \fdiv(f)(v) = 
        \sum_{i=1}^r s_i = (s_1 + s_r) + \sum_{i = 2}^{r-1} s_i < -(r - 2)
        = -K(v).
    \end{equation*}
    This is absurd, as this would imply $(K + \fdiv(f))(v) < 0$.
\end{proof}

\begin{proposition}
    \label{prop:no-disjoint-realizable}
    Let $D = K + \fdiv(f)$ be an effective canonical divisor.
    If there are no two disjoint horizontal cycles
    (possibly of different heights),
    then $D$ is realizable.
\end{proposition}

\begin{proof}
    We need to verify the conditions of Theorem \ref{thm:realizability-iff} are
    satisfied.

    If we have an inconvenient vertex $v$, by Lemma
    \ref{lemma:inconvenient-slopes} there are (at least)
    two edges $e_1, e_2$
    emanating from $v$ such that the outgoing slope of $f$ along these edges is
    strictly positive.
    Select two simple ($=$ not self-intersecting) paths $\gamma_1, \gamma_2$
    (seen as functions $[0, 1] \to \Gamma$ with $\gamma_i(0) = v$)
    along $e_1, e_2$, that are maximal for the property that $f \circ \gamma_i$
    is non-decreasing and $f \circ \gamma_i$ is strictly increasing on $(1 -
    \epsilon, 1)$ for some $\epsilon > 0$. Let $x_i := \gamma_i(1)$.
    Let $C_i$ be the connected component
    of $f^{-1}(f(x_i))$ containing $x_i$. By construction, $C_i$ is a
    local maximum, so by Proposition \ref{prop:local-max}, the $C_i$ each
    contain a horizontal cycle. Since by assumption these cycles have to
    intersect, we obtain $C_1 = C_2$.

    Now, if $\gamma_1, \gamma_2$ intersect, we can choose $(t_1, t_2)$ a pair
    such that $\gamma_1(t_1) = \gamma_2(t_2)$,
    minimal for the partial order $(t_1, t_2) \leq (s_1, s_2) \iff
    t_1 \leq s_1$ and $t_2 \leq s_2$. Then $\gamma_1|_{[0, t_1]} \oplus
    \overleftarrow{\gamma_2|_{[0, t_2]}}$ is a simple cycle that lies above $v$.
    Here by $\overleftarrow{\gamma}$ we mean the reversed path
    $\overleftarrow\gamma(t) = \gamma(1 - t)$. By $\oplus$ we mean the
    concatenation of
    paths, that is.
    \begin{equation*}
        \alpha\oplus\beta(t) := 
        \begin{cases}
            \alpha(2t) &\textrm{if } t \leq 1/2,\\
            \beta(2t - 1) &\textrm{if } t \geq 1/2.
        \end{cases}
    \end{equation*}
    Lastly, when we write the restriction $\gamma|_{[a, b]}$, it is understood
    that this new path is reparametrized as to have again domain $[0, 1]$, that
    is
    \begin{equation*}
        \gamma|_{[a, b]}(t) =
            \gamma(a + t(b-a))
    \end{equation*}

    If $\gamma_1, \gamma_2$ don't intersect, since $C_1 = C_2$ is connected,
    there is a simple path $\tau$ from $x_1$ to $x_2$.
    Then the path $\gamma_1 \oplus
    \tau \oplus \overleftarrow{\gamma_2}$ is a simple cycle that lies above $v$.
    We conclude that every inconvenient vertex is contained in a simple cycle
    that lies above it.

    Now let $e$ be a horizontal edge between two vertices $v_1$ and $v_2$.
    Let $C$ be the connected component of $f^{-1}(f(e))$ containing $e$.
    If $C \setminus e$ is connected, then
    there exists a simple path $\gamma$
    in $C \setminus e$ from $v_1$ to $v_2$. Then going along $\gamma$ from $v_1$
    to $v_2$ and then from $v_2$ to $v_1$ along $e$ determines a horizontal
    simple cycle containing $e$. Hence $e$ is not a problematic horizontal edge.

    So suppose $C \setminus e$ is 
    disconnected and let $C_1$, $C_2$ be the two components containing the
    vertices $v_1$, resp. $v_2$.
    We will show that there exist simple paths
    $\gamma_i$ in $\Gamma \setminus e$ from
    $v_i$ to the same horizontal cycle and like before,
    this would prove that $e$ is contained in
    in a simple cycle that lies above it.

    If $C_i$ has any point $x$ that is a leaf of $C$,
    then $x$ has an adjacent edge on which $f$ has
    strictly positive outgoing slope. Indeed, this follows because
    all of the other $\val_\Gamma(x) - 1$ edges have non-zero slopes,
    and if they
    were all negative, then $\fdiv(f)(x) \leq -(\val_\Gamma(x) - 1)$,
    but this would
    imply that $(K + \fdiv(f))(x) \leq -1$,
    which is absurd as we assumed $D = K +
    \fdiv(f)$ is effective.
    So $x$ neighbours an edge with strictly positive outgoing slope,
    and like before, we could take a path $\gamma$
    starting at $x$ along this edge, which is
    maximal for the property that $f \circ \gamma$ is non-decreasing and
    $f \circ \gamma$ is strictly increasing on $(1 - \epsilon, 1)$ for some
    $\epsilon$. Then $\gamma(1)$ would lie on a local maximum which
    contains a distinguished cycle.

    So suppose $C_i$ contains no leaf of $C$, then the only leaf of $C_i$ is
    possibly $v_i$.
    If we denote $V_i$ the vertices of $C_i$ and $E_i$ the edges of $C_i$,
    we know that $\sum_{v \in V_i} \val_{C_i}(v) = 2\cdot\#E_i$.
    It follows that the sum is even and so if $v_i$ were of valence 1 in $C_i$,
    there would also need to be another vertex of odd valence,
    and so this one would
    have to be of valence at least 3.
    In any case, we get that $\#E_i \geq \#V_i$ and so
    $g(C_i) \geq 1$. In other words, $C_i$ contains a simple cycle.
    The two cycles contained in $C_1, C_2$ are both horizontal, so they have
    to intersect, but this contradicts the fact that $C_1, C_2$ are disjoint.

    So we conclude that $e$ is contained in a simple cycle that lies above it,
    and so having verified the conditions of
    Theorem \ref{thm:realizability-iff}, we conclude $D$ is realizable.
\end{proof}

\begin{corollary}
    \label{cor:no-disjoint-realizable}
    Suppose $\Gamma$ does not contain disjoint cycles, then every $D \in |K|$ is
    realizable.
\end{corollary}

\begin{corollary}
    If $|K|$ is of dimension at most $g - 1$,
    then every $D \in |K|$ is realizable.
\end{corollary}

\begin{proof}
    If $\Gamma$ contained disjoint cycles, then $|K|$ would contain a cell of
    dimension at least $g$ by Proposition \ref{prop:disjoint-dim}. The
    corollary then follows from the contrapositive of this statement.
\end{proof}

\begin{remark}
    The converse of Proposition \ref{prop:no-disjoint-realizable} is not true.
    For example, the canonical divisor is realizable for the graph obtained by
    joining two cycles by a pair of edges (see Figure
    \ref{fig:realizable-canonical}).
    On the other hand, the converse of Corollary 
    \ref{cor:no-disjoint-realizable} is true
    as the following proposition shows.
\end{remark}

\begin{figure}[ht]
    \centering
    \begin{tikzpicture}[scale=2.5]
        \draw (0,0)--(1,0);
        \draw (0,1)--(1,1);
        \draw (0,0) to[bend right] (0, 1); 
        \draw (0,0) to[bend left] (0,1);
        \draw (1,0) to[bend right] (1,1);
        \draw (1,0) to[bend left] (1,1);
        \begin{scope}[every node/.style=divisor]
            \node at (0,0) {};
            \node at (0,1) {};
            \node at (1,0) {};
            \node at (1,1) {};
        \end{scope} 
    \end{tikzpicture}
    \caption{Graph with realizable canonical divisor.}
    \label{fig:realizable-canonical}
\end{figure}
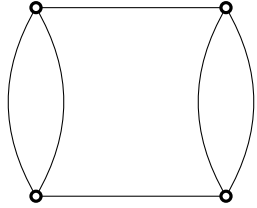

\begin{proposition}
    If $\Gamma$ contains disjoint cycles, then there exist non-realizable
    divisors $D \in |K|$.
\end{proposition}

\begin{proof}
    Let $Z_1, Z_2 \subseteq \Gamma$ be the two disjoint cycles.
    Let also $\gamma \subseteq \Gamma$ be a simple path such that
    $\gamma(0) \in Z_1$, $\gamma(1) \in Z_2$ and 
    $\gamma((0, 1)) \cap (Z_1 \cup Z_2) = \emptyset$.
    
    The union $Z = Z_1 \cup Z_2 \cup \gamma$ is a closed subgraph where each
    vertex $v$ has at most $\val(v) - 2$ outgoing edges, this means $Z$ can
    fire (if $Z = \Gamma$, firing doesn't have any effect). So fire $Z$ by a
    small amount to
    obtain a divisor $D = K + \fdiv(f) \in |K|$.
    Now, since $Z_1, Z_2$ are disjoint, $\gamma$ is non-trivial and so passes
    through at least one edge $e$. By construction, $e$ is a horizontal edge.
    If $C \subseteq \Gamma$ is any simple cycle passing through $e$, $C$ must
    intersect the complement of $Z$. Indeed, removing $e$ from $Z$ would
    disconnect it, as $e \subseteq \gamma$, which is a simple path connecting
    the two disjoint cycles. But because $D$ was obtained by firing $Z$, it is
    clear that $f(\Gamma \setminus Z) < f(Z)$, and so $C$ is not a cycle that
    lies above $e$. We conclude that $D$ is not realizable.
\end{proof}

\begin{corollary}
    There exists a maximal cell $\sigma$ of $|K|$ of dimension $g-1$ such that
    $\sigma \subseteq \Real(|K|)$.
\end{corollary}

\begin{proof}
    Let $\sigma$ be the cell of dimension $g-1$ given by Proposition
    \ref{prop:minimal-dim-attained}. Let $D = K + \fdiv(f) \in \relint(\sigma)$,
    so that
    $\supp D \cap V = \emptyset$.
    It follows that that
    \begin{equation*}
        g - 1 = \dim \sigma = \deg D - (g - g(\Gamma \setminus D))
    \end{equation*}
    and so $g(\Gamma \setminus D) = 1$. 
    We have that $\supp D \subseteq \bend f$.
    In particular, 
    \[g(\Gamma \setminus \bend f) \leq g(\Gamma \setminus \supp D) = 1,\]
    and so there are no two
    disjoint horizontal cycles. This implies by Proposition
    \ref{prop:no-disjoint-realizable} that $D$ is realizable.
    Since $D$ was arbitrary, it follows that
    $\relint(\sigma) \subseteq \Real(|K|)$, and since $\Real(|K|)$ is closed by
    Proposition \ref{prop:rl-closed}, we deduce that
    $\sigma \subseteq \Real(|K|)$.
\end{proof}

\subsection{Specialization of linear series}
\label{sec:lin-series}

We are also interested in the specialization of linear series. Recall that given
a smooth projective curve $X$ over a valued field $K$, which is geometrically
connected, and of genus $\geq 2$,
we have defined the specialization map
\[\rho: \Div(X_\closure{K}) \to \Div_\Q(\Gamma).\]
First, note that
the specialization map $\rho$ by definition preserves the property
of being effective. We claim that it also preserves the property of being
principal. Let $D$ be a principal divisor on $X_\closure{K}$, then since $D$ is
a Weil divisor (Weil and Cartier divisors are identified on regular schemes),
it may be written as a finite sum
\[\sum_{i = 1}^r n_iP_i\]
where the $P_i \in X_\closure{K}$ are $\closure{K}$-rational points.
But then there
exists an algebraic extension $L$ of $K$ such that $P_i$ are all $L$-rational.
In particular $D$ may be seen as a principal divisor on $X_L$, and hence it
determines a rational function $f$ on $X_L$.
Let $\fX$ be the minimal regular model of $X_L$, with $C_1, \dots, C_s$
corresponding to a set of vertices $v_1, \dots, v_s$ on the metric graph
$\Gamma$.
Since $X_L$ is open in $\fX$,
$f$ determines a rational function on all of $\fX$, and so a principal
divisor $D'$ on $\fX$ (again we identify Weil and Cartier divisors, since $\fX$
is regular). Now, $f$ could only acquire new zeroes and poles on the complement
of $X_L$ in $\fX$, so only on the special fiber. It follows that the difference
\[D' - \sum_{i = 1}^r n_i \closure{P_i}\]
is a vertical divisor, which we will write as
\[D_v = \sum_{i = 1}^s m_i C_i.\]

Now, note that we may rewrite the restriction $\rho: X_L \to \Div(\Gamma)$ as
\[\rho(P) = \sum_{i = 1}^s (\closure{P} \cdot C_i)v_i,\]
where $\closure{P} \cdot C_i$ denotes the intersection number
$\deg(\cO_\fX(\closure{P})|_{C_i})$.
Let 
\[\closure{D} = \sum_{i = 1}^rn_i\closure{P_i},\]
then it follows by
bilinearity
of the intersection number that
\[\rho(D) = \sum_{i = 1}^s (\closure{D} \cdot C_i)v_i.\]
Since $D'$ is principal, we have that $D' \cdot C_i = 0$ and hence
\[\closure{D} \cdot C_i = -D_v \cdot C_i.\]
Now, note that for any component $C_j$, the divisor
\[\sum_{i = 1}^s (C_j \cdot C_i) v_i\]
corresponds to the divisor obtained by firing the vertex $v_j$ by the unit
distance (distance between any adjacent vertices of $\Gamma$), in particular it
is principal. We conclude that 
$\rho(D)$ is a sum of principal divisors, and so also principal.

Now, let $|D|$ be a complete linear series on $X$, then for any divisor
$D' \in |D|$, we have that $D'$ is effective and $D - D'$ is principal.
By what we just showed, this implies that $\rho(D')$ is effective and 
$\rho(D - D')$ is principal. As a result, $\rho(D') \in |\rho(D)|$ and so
$\rho(|D|) \subseteq |\rho(D)|$ (this is
actually the main ingredient for the proof of the Specialization Lemma).
We are now interested in studying the specialization of linear series
$\fd \subseteq |D|$, which we will define as the topological closure
of the subspace $\rho(\fd) \subseteq |\rho(D)|$ and we will denote it
by $\trop(\fd)$.

To distinguish when we talk about linear systems on the algebraic curve or on
the tropical curve, we will denote the objects on the algebraic curve with a
subscript. So in what follows,
fix a divisor $D_X \in \Div(X_\closure{K})$, specializing to
a divisor $D \in \Div(\Gamma)$.

\begin{proposition}
    \label{prop:fjp-props}
    Let $\fd_X \subseteq |D_X|$ be a linear series of rank $r$. Then
    $\trop(\fd_X)$ is a finitely generated tropical linear series
    of rank at least $r$.
\end{proposition}
\begin{proof}
    This is shown in Lemmas 6.1, 6.2 and Proposition 6.4 of
    \cite{kodaira-dimensions}.
\end{proof}

\begin{remark}
    Note that in my thesis the term ``tropical linear series" does not designate
    the same thing as the same term in \cite{kodaira-dimensions} and
    \cite{linsys-independence}. Here it means a tropically convex subset of a
    complete linear series (in analogy to the nomenclature from algebraic
    geometry).
\end{remark}

\begin{corollary}
    For $\fd_X \subseteq |D_X|$ a linear series, $\trop(\fd_X)$
    admits the structure
    of an abstract polyhedral complex.
\end{corollary}

\begin{proof}
    By Proposition \ref{prop:fjp-props}, $\trop(\fd_X)$ is finitely generated
    and so
    this follows from Proposition \ref{prop:fin-gen-definable}
    and Corollary \ref{cor:closed-def-complex}
\end{proof}

A linear system $\fd_X$ is a projective subspace of $|D_X|$, in particular it
has a well-defined notion of independence -- a set of vectors of $\fd_X$ of size
$s$ is independent if and only if it is not contained in a projective subspace
of $\fd_X$ of dimension $s-1$.
This gives $\fd_X$ the structure of a matroid and it turns out that this
notion translates well through tropicalization.

\begin{definition}
    Let $S = \{D + \fdiv(\phi_1), \dots, D + \fdiv(\phi_n)\}$
    be a subset of
    $|D|$. We say that $S$ is \emph{tropically dependent}
    if there are real numbers
    $a_i$ such that for every point $v \in \Gamma$, the minimum in 
    $\min_i\{\phi_i(v) + a_i\}$ is achieved at least twice.
\end{definition}

It turns out that linearly dependent subsets of $|D_X|$
specialize to tropically dependent subsets of $|D|$. For $\fd_X \subseteq |D_X|$
a linear series, its rank is equal to its dimension as a projective subspace of
$|D_X|$ and so any subset of $r + 2$ points of $\fd_X$ is linearly dependent.
Hence we expect the same property to hold after tropicalizing.

\begin{proposition} 
    \cite[Lemma 6.2]{kodaira-dimensions}
    Any subset of $\trop(\fd_X)$ of size at least $r+2$ is
    tropically dependent.
\end{proposition}

The same way the notion of rank provides a lower bound on the dimension of a
tropical linear series, the notion of tropical independence yields an upper
bound.

\begin{proposition}
    \label{prop:max-dim}
    \cite[Corollary 4.7]{linsys-independence} Let $\fd \subseteq |D|$ be a
    finitely generated submodule such that any set of $r + 2$ functions of
    $\fd$ is
    tropically dependent, then
    $\dim \fd \leq r$.
\end{proposition}

\begin{corollary}
    If $\fd \subseteq |D|$ is a finitely generated tropical linear series
    of rank $\geq r$, such that any set of $r + 2$ points of $\fd$ is
    tropically dependent, then $\fd$ is of rank $r$. Furthermore,
    $\fd$ is equi-dimensional of dimension $r$ (in the sense that all the maximal
    cells have dimension $r$).
\end{corollary}

\begin{proof}
    Corollary \ref{cor:fin-gen-dim} implies that
    all the maximal faces of $\fd$ have dimension at least $r(\fd)$,
    and so
    $\dim \fd \geq r(\fd) \geq r$. But by Proposition
    \ref{prop:max-dim} this forces $\dim \fd = r$ and so in particular
    $r(\fd) = r$, and the maximal dimensional faces of $\fd$ need to
    have dimension exactly $r$.
\end{proof}

\begin{corollary}
    If $\fd_X \subseteq |D_X|$ is a linear series of rank $r$,
    then $\trop(\fd_X)$
    is of rank $r$, and equi-dimensional of dimension $r$.
\end{corollary}

\section{Discrete representations}
\label{sec:discrete-representations}

In this section, we are going to briefly discuss possible approaches to
working with tropical curves using computer techniques.

One approach, which is the closest to the original definitions is to
represent a metric graph as a set of vertices $V$ and edges $E$ with a fixed
orientation and length.
Divisors and rational functions can be entirely determined
by specifying their value on a finite set of points of the graph $\Gamma$.
Furthermore, a point of $\Gamma$ is either a vertex, in which case it
corresponds to an element of $V$, or a point along an edge,
in which case it may be specified by an edge $e \in E$, and a distance along
that edge. This approach is well-suited for working with tropical submodules, as
the tropical operations are fairly easy to implement.

Another approach is to restrict our attention to rational functions and divisors
on a given model of $\Gamma$, for which each edge has identical length. The
advantage of this approach is that operations on such a curve may be represented
by matrix operations, which can make many things significantly faster and easier
to implement. It is much easier to represent and work with chip-firing moves,
find $v$-reduced divisors and
go back and forth between divisors and rational functions. An important aspect
of using this representation is that in a linear system there may only be a
finite number of divisors which are supported on a given model of $\Gamma$. This
allows us to develop an algorithm for finding the set of these divisors
exhaustively.

We are now going to describe in detail how to work with linear systems on graphs
and in particular justify how these finite models carry
information about linear systems on the whole metric graphs.

\subsection{Graphs}
\label{sec:graphs}

Let $\Gamma$ be a metric graph. We are mainly interested in metric graphs with
sides of integer length as these are those that appear as the tropicalization of
an algebraic curve. We can choose the model $G = (V, E)$ of $\Gamma$ that has
all edge lengths of size 1. If we restrict rational maps and divisors to be
supported on $V$, we have a nice description.

\begin{definition}
    Let $G = (V, E)$ be a graph. A \emph{level map} is a function
    \begin{equation*}
        f: V \to \Z.
    \end{equation*}
\end{definition}

A level map $f$ uniquely determines a rational function on $\Gamma$ if we
interpolate linearly between the vertices. Indeed, as we assumed all edge
lengths are of size 1, this implies the resulting function has integral slopes.

Reciprocally, up to adding a constant, we may assume
a rational function $f: \Gamma \to \R$ supported on $V$ admits integral values
on $V$. So by restricting $f$ to $V$ this determines a level map up to a
constant.

\begin{definition}
    Let $G = (V, E)$ be a graph. A \emph{divisor} on $G$ is a divisor $D$ 
    on $\Gamma$, which is supported on the set of vertices $V$.
    We denote the set of divisors on $G$ by $\Div(G)$.
\end{definition}

When we scale a metric graph by an integer, we do not affect the combinatorial
structure of the metric graph, nor the structure of linear systems.
Scaling the metric graph $m$ times has the effect
of uniformly subdividing the associated model where all edges have length 1
by splitting each edge into $m$ edges.
Let $D \in
\Div(\Gamma)$ be a $\Q$-rational divisor. Then that all the points in the
support of $D$ are a rational distance away from any given vertex.
If we take the common denominator of these numbers, say $d$, we know that 
$D$ will be supported on the $d$\textsuperscript{th}
subdivision of the graph $G$.
So we deduce that by subdividing the graph $G$ we may obtain a more and more
faithful representation of the whole set of divisors $\Div(\Gamma)$.

\begin{definition}
    Let $G = (V, E)$ be a graph. We define the \emph{incidence map} to be the
    map
    \begin{equation*}
        \phi: E \to \{\{x, y\}: x, y \in V\},
    \end{equation*}
    which maps each edge to its respective vertices.
    
    For $x \in V$, denote $E_x = \{e \in E: x\in \phi(E)\}$ the set of edges
    adjacent to $x$ and for
    $e \in E_x$, with $\phi(e) = \{x, y\}$, we denote $\nu_x(e) = y$ the vertex
    adjacent to $x$ along the edge $e$.
\end{definition}

\begin{definition}
    Let $f$ be a level map on $G$. Then the \emph{order} of $f$ at $x$ is
    defined by 
    \begin{equation*}
        \ord_x(f) = \sum_{e \in E_x}(f(\nu_x(e)) - f(x)).
    \end{equation*}
    The divisor $\fdiv(f) \in \Div(G)$
    associated to $f$ is defined by
    \begin{equation*}
        \fdiv(f) = \sum_{x \in V}\ord_x(f)\cdot x
    \end{equation*}
\end{definition}

\begin{remark}
    This definition is compatible with the analogous definitions on rational
    maps.
\end{remark}

\begin{definition}
    Let $A \subseteq V$ be any subset. The level map associated to a chip firing
    move is
    \begin{equation*}
        CF(A)(x) := 
        \begin{cases}
            1 &\textrm{if } x \in A\\
            0 &\textrm{otherwise}
        \end{cases}
    \end{equation*} 
\end{definition}

\begin{remark}
    The rational map associated to this chip firing move on a metric graph is
    $\sum_{x \in A} CF(\{x\}, 1)$, or equivalently if $Z \subseteq \Gamma$
    is the subgraph obtained by adding all the edges between the vertices in
    $A$, then the chip firing move is the same as $CF(Z, 1)$. We say $Z$ is the
    subgraph \emph{spanned} by the set of vertices $A$.
\end{remark}

\begin{remark}
    Any level map is a sum of chip-firing moves (up to a constant).
\end{remark}

The advantage of working with such a discretization is that we can express
level maps and divisors as vectors, and chip-firing moves as matrix operations.
To this end, fix an ordering $\{v_1, \dots, v_n\}$ on the vertices $V$.

\begin{definition}
    Let $f$ be a level map, the \emph{vector associated to} $f$, denoted by
    $[f]$ is the column-vector
    \begin{equation*}
        [f] = 
        \begin{bmatrix}
            f(v_1)\\
            \vdots\\
            f(v_n)
        \end{bmatrix}.
    \end{equation*}
    Similarly, for $D$ a divisor, the \emph{vector associated to} $D$ is
    \begin{equation*}
        [D] =
        \begin{bmatrix}
            D(v_1)\\
            \vdots\\
            D(v_n)
        \end{bmatrix}
    \end{equation*}
\end{definition}

\begin{definition}
    We define the \emph{adjacency matrix} of $G$ to be the symmetric
    matrix
    $\Adj(G) \in \Z_{\geq 0}^{n\times n}$ defined by
    \begin{equation*}
        \Adj(G)_{i, j} = \#\{e \in E: \phi(e) = \{v_i, v_j\}\}
    \end{equation*}

    We define also the \emph{firing matrix} of $G$ to be the matrix
    \begin{equation*}
        F = F(G) = \Adj(G) - \diag(\val(v_1), \dots, \val(v_n))
    \end{equation*}
\end{definition}

\begin{proposition}
    For $f$ a level map, we have that
    \begin{equation*}
        [\fdiv(f)] = F\cdot[f]
    \end{equation*}
\end{proposition}

\begin{proof}
    This follows from the observation that
    \begin{equation*}
        \ord_x(f) = \sum_{e \in E_x} f(\nu_x(e)) - \val(x)f(x)
    \end{equation*}
\end{proof}

\begin{remark}
    Let $D$ be a divisor. To say that $D$ is principal
    is the same as saying that there exists some
    $\vec{u} \in \Z^{n \times n}$ such that
    \begin{equation*}
        [D] = F\cdot \vec{u}.
    \end{equation*}
    In other words, $D$ is principal if and only if $[D] \in \im(F)$.

    Let $F^+$ be the pseudo-inverse of $F$.
    By the properties of the pseudo-inverse, $FF^+$ is the projection on the
    image of $F$. It follows that when $D$ is principal,
    \begin{equation*}
        FF^+ [D] = [D].
    \end{equation*} 
    Let $g$ be the level map given by the vector $F^+[D]$, we deduce that
    $D = \fdiv(g)$.
    So when $D$ is principal,
    this gives us a way to find a level map whose associated divisor is $D$.
\end{remark}

\begin{definition}
    A \emph{path} $\gamma$ in $G$ is a sequence of
    vertices $x_0, x_1, \dots, x_n$ such
    that for each $i \in \{1, \dots, n\}$, there is an edge in $G$
    between the vertices $x_{i-1}$ and $x_i$.
    The \emph{length} of the path $\gamma$ is $L(\gamma) = n$.

    We define the distance between two vertices $x, y$ on the graph to be
    \begin{equation*}
        d(x, y) := \inf L(\gamma).
    \end{equation*}
\end{definition}

\begin{remark}
    The length of a path and distance between points clearly 
    agrees with the notions 
    defined for metric graphs.
\end{remark}
\begin{remark}
    For a given vertex $x$, an efficient way to find the distance of $x$ from
    each other
    vertex is via a single pass of breadth-first search (BFS).
\end{remark}

\begin{definition}
    Let $A \subseteq V$ be a set of vertices. We say $A$ is \emph{connected}
    when for all $x, y \in A$, there exists a path from $x$ to $y$ contained in
    $A$.
\end{definition}

\subsection{Reduced divisors on graphs}
\label{sec:reduced-div-graphs}

Unless stated otherwise, we suppose $G$ is a connected graph.

\begin{proposition}
    A divisor
    $D \in \Div(G)$ is $v$-reduced if and only if it
    is effective away from $v$ and 
    for all $A \subset V$ with
    $v \notin A$, we have that $A$ cannot fire with respect to $D$.
\end{proposition}

\begin{proof}
   Suppose that $D$ satisfies these assumptions.
   If $Z \subseteq \Gamma$ is a subgraph that can fire with respect
   to, then $\boundary Z \subseteq \supp D$. 
   But then the set of vertices $V\cap Z$
   can fire on $G$, which implies that $v \notin V\cap Z$,
   and so in particular $v \notin Z$, which shows that $D$ is $v$-reduced. 

   Reciprocally, suppose $D$ is $v$-reduced. Let $A$ be a subset of $V$ that can
   fire on $G$. If we let $Z$ be the subgraph spanned by $A$, then $Z$ can fire.
   So by assumption $v \notin Z$ and so in particular $v \notin A$.
\end{proof}

We will now describe an algorithm that can be used to find $v$-reduced divisors.

\begin{lemma}
    \label{lemma:existence-effective}
    Let $D \in \Div(G)$. There exists a divisor on $G$
    effective away from $v$ that is
    linearly equivalent to $D$.
\end{lemma}

\begin{proof}
    Let $n = \max_{x\in V} d(x, v)$ and for $i \in \{0, n\}$ let
    \begin{equation*}
        A_i := \{x \in V: d(v, x) \leq i\}.
    \end{equation*} 
    We will proceed by induction to define $D_i$ such that
    $D_i$ is effective away from $A_{n-i}$ and $D_i$ is linearly equivalent to
    $D$. In particular $D_n$ will be effective away from $A_0 = \{v\}$ and this
    would show the lemma.

    We start with $D_0 = D$. Suppose we have found $D_{i-1}$ for some $i$,
    then let 
    \begin{equation*}        
        m = \min\{D_{i-1}(x): x\in A_{n-i+1} \setminus A_{n-i}\}.
    \end{equation*}
    We set $D_i = D_{i-1} + \fdiv (m \cdot CF(A_{n-i}))$,
    then $D_i$ is effective away from $A_{n-i}$. Indeed
    for all $x \in A_{n-i+1} \setminus A_{n-i}$, there is at least one edge
    connecting $x$ to a vertex in $A_{n-i}$ and so by firing $A_{n-i}$,
    $m$ chips are moved to $x$ along this edge.
    In particular $D_i(x) \geq 0$. Furthermore, it is clear that the vertices in
    $V\setminus A_{n-i+1}$ are not affected, hence $D_i$ is indeed effective
    away from $A_{n-i}$.
\end{proof}

\begin{proposition}
    \label{prop:existence-v-reduced-graphs}
    Let $D \in \Div(G)$. There exists a unique $v$-reduced divisor on
    $G$ linearly
    equivalent to $D$.
\end{proposition}

\begin{proof}
    By Lemma \ref{lemma:existence-effective}, we may assume that $D$ is
    effective away from $v$.

    Suppose $D$ is not $v$-reduced.
    Let $A$ be the maximal subset of $\Gamma \setminus \{v\}$ that can fire.
    There is a unique such subset, since when $A_1, A_2$ are two subsets
    that can fire, then
    firing $A_1 \cup A_2$ corresponds to the chip-firing move
    $CF(A_1) \oplus CF(A_2)$. Since $R(D)$ is a tropical semi-module,
    we deduce that $A_1 \cup A_2$ can fire.
    Let $m$ be maximal for the property that $D + \fdiv(m\cdot CF(A))$ is 
    effective away from $v$ and set $D' = D + \fdiv(m\cdot CF(A))$.
    If $D'$ is not $v$-reduced, we may repeat this procedure until we
    obtain a $v$-reduced divisor.

    It remains to check that this algorithm terminates.
    Define as earlier
    \begin{equation*}
        B_i := \{x \in V: d(v, x) \leq i\}.
    \end{equation*} 
    We define the partial order $\prec$ on the set of divisors on $G$
    given by $D_1 \prec D_2$ if and only
    if there exists some $n \in \N$ such that
    $\deg D_1|_{B_i} = \deg D_2|_{B_i}$ for $i < n$
    and $\deg D_1|_{B_n} < \deg D_2|_{B_n}$.

    Let $D$ and $D'$ as before, we claim that $D \prec D'$.
    Indeed, let $A$ be the subset from before and let 
    $n$ maximal for the property that $A \cap B_n = \emptyset$.
    It follows that there is some $x \in B_{n+1} \cap A$ and so by the
    definition of distance there is some $y \in B_n$ that is adjacent to $x$.
    Firing $A$ will move at least one chip from $x$ to $y$ and since no chip is
    moved away from $B_n$ when firing $A$, we deduce that
    $\deg D|_{B_n} < \deg D|_{B_n}$. Clearly, firing $A$ does not affect the
    vertices in $B_i$ for $i < n$, so we deduce that $D \prec D'$.
    Since $D \prec D'$ and there is only a finite number of divisors of a given
    degree, we deduce that the algorithm has to terminate.

    Unicity then follows from the unicity of $v$-reduced divisors on
    metric graphs.
\end{proof}

\begin{remark}
    We can find the set $A$ from the above proof using
    \emph{Dhar's burning algorithm}.

    Start by distributing chips on the vertices, according to the divisor $D$,
    and light a fire at $v$.
    Then repeat the following steps:
    \begin{enumerate}
        \item If there is an unburned edge adjacent to a burned vertex, burn it.
            Otherwise terminate the algorithm.
        \item If there is at least one chip on the vertex adjacent to the
            corresponding edge, remove one chip. Otherwise burn the vertex.
        \item Go back to step 1.
    \end{enumerate}
    The set $A$ of unburned vertices is our desired set.
\end{remark}

\begin{remark}
    It is clear from the proof of the proposition that we may also easily find
    the corresponding level map such that $D + \fdiv(f)$ is $v$-reduced.
    When $D \sim D'$, and we know that $D + \fdiv(f)$ and $D' + \fdiv(g)$
    are $v$-reduced, then in fact $D + \fdiv(f) = D' + \fdiv(g)$ by unicity.
    It follows that $D - D' = \fdiv(g - f)$ and so this gives us another way to
    find level maps that relate two linearly equivalent divisors.
\end{remark}

\begin{corollary}
    Suppose $\Gamma$ is a graph and $G$ is a model with equal edge lengths.
    If $D \in \Div(\Gamma)$ is a divisor supported on the vertices of $G$,
    and $v$ is a vertex,
    then the $v$-reduced divisor linearly equivalent to $D$ is 
    also supported on the vertices.
\end{corollary}

\begin{proof}
    This follows from Proposition \ref{prop:existence-v-reduced-graphs} and
    the unicity of $v$-reduced divisors on metric graphs.
\end{proof}

\begin{proposition}
    Let $D$ be a $v$-reduced divisor on $G$ and let $f$ be a
    level map such that
    $D + \fdiv(f) \geq 0$. Let $v \in A \subset V$ and
    \begin{equation*}
        m = \max\{f(x): x \in A, \degout_A(x) > 0\},
    \end{equation*}
    then $f(V \setminus A) \leq m$.
\end{proposition}

\begin{proof}
    Then let $B$ the subset of $V\setminus A$ on which
    $f|_{V\setminus A}$ attains its maximum
    and suppose for the sake of contradiction that $f(B) > m$.
    Let $Z \subseteq \Gamma$ be the subgraph spanned by $B$. Then $Z$
    is a local maximum for $f$. Indeed, for any vertex $x \notin B$
    adjacent to a point
    in $B$, we have that either $x \in V\setminus A$, in which case
    $f(x) < f(B)$ by assumption, or $x \in A$.
    In this case, since we assumed $x$ is adjacent to a point in $B$, we have
    that $\degout_A(x) > 0$ and hence by definition of $m$, 
    $f(x) \leq m < f(B)$. So $Z$ is indeed a local maximum but this implies
    by Proposition \ref{prop:reduced-local-max} that
    $v\in Z$, a contradiction as $v \in A$.
\end{proof}

\begin{remark}
    This is the same as saying that the subgraph spanned by the vertices in
    $A$ is connected.
\end{remark}

\begin{proposition}
    Let $D$ be a $v$-reduced divisor on $G$ and $f$ a level map such that
    $D + \fdiv(f) \geq 0$.
    For any $x \in V$, there exists a non-decreasing path from $x$ to $v$.
\end{proposition}

\begin{proof}
    We will proceed by induction on $f(v) - f(x)$.
    If $f(v) = f(x)$, by Corollary \ref{cor:superlevel-connected},
    $f^{-1}(f(v))$
    is connected and so there exists a path from $x$ to $v$.

    If $f(v) > f(x)$, then we know that the set
    $f^{-1}([f(x), \infty))$ is connected 
    by Corollary \ref{cor:superlevel-connected}. Let $y$ be the closest vertex
    to $x$ such that $f(y) > f(x)$ and let $x = x_0, x_1, \dots, x_n = y$ be a
    minimal path from $x$ to $y$. This path is clearly non-decreasing as
    $f(x_i) = f(x)$ for all $i < n$ by assumption.
    By induction, there exists a non-decreasing 
    path from $y$ to $v$. By concatenating these paths we obtain a
    non-decreasing path from $x$ to $v$.
\end{proof}

\nocite{*}
\printbibliography

@misc{linsys,
    author = {Christian Haase and Gregg Musiker and Josephine Yu},
    primaryclass = {math.AG},
    year = {2009},
    title = {Linear Systems on Tropical Curves},
    eprint = {0909.3685},
    archiveprefix = {arXiv},
}

@inbook{tropical-hodge,
    doi = {10.1007/978-1-4939-7486-3_16},
    publisher = {Springer New York},
    url = {http://dx.doi.org/10.1007/978-1-4939-7486-3_16},
    pages = {353–368},
    isbn = {9781493974863},
    issn = {2194-1564},
    author = {Lin, Bo and Ulirsch, Martin},
    year = {2017},
    title = {Towards a Tropical Hodge Bundle},
    booktitle = {Combinatorial Algebraic Geometry},
}

@misc{amini-reduced,
    year = {2012},
    eprint = {1007.5364},
    author = {Omid Amini},
    archiveprefix = {arXiv},
    title = {Reduced Divisors and Embeddings of Tropical Curves},
    primaryclass = {math.CO},
}

@misc{realizability-canonical,
    year = {2017},
    author = {Martin Moeller and Martin Ulirsch and Annette Werner},
    primaryclass = {math.AG},
    title = {Realizability of tropical canonical divisors},
    eprint = {1710.06401},
    archiveprefix = {arXiv},
}

@misc{realizability-pluricanonical,
    archiveprefix = {arXiv},
    author = {Felix Röhrle and Johannes Schwab},
    eprint = {2109.03579},
    year = {2021},
    title = {Realizability of tropical pluri-canonical divisors},
    primaryclass = {math.AG},
}

@misc{trop-rr,
    primaryclass = {math.CO},
    title = {Riemann-Roch and Abel-Jacobi theory on a finite graph},
    archiveprefix = {arXiv},
    year = {2007},
    author = {Matthew Baker and Serguei Norine},
    eprint = {math/0608360},
}

@misc{specialization-lemma,
    author = {Matthew Baker},
    archiveprefix = {arXiv},
    year = {2007},
    eprint = {math/0701075},
    primaryclass = {math.NT},
    title = {Specialization of linear systems from curves to graphs},
}

@misc{metric-graph,
    title = {What is actually a metric graph?},
    eprint = {1912.07549},
    year = {2021},
    primaryclass = {math.CO},
    author = {Delio Mugnolo},
    archiveprefix = {arXiv},
}

@book{metric-geometry,
    year = {2001},
    title = {A Course in Metric Geometry},
    publisher = {American Mathematical Society},
    author = {Dmitri Burago and Yuri Burago and Sergei Ivanov},
}

@book{liu,
    title = {Algebraic Geometry and Arithmetic Curves},
    year = {2002},
    author = {Qing Liu},
    publisher = {Oxford University Press},
}

@article{stable-reduction,
    url = {http://dx.doi.org/10.1007/BF02684599},
    title = {The irreducibility of the space of curves of given genus},
    volume = {36},
    issn = {1618-1913},
    author = {Deligne,  P. and Mumford,  D.},
    number = {1},
    year = {1969},
    pages = {75–109},
    journal = {Publications mathématiques de l’IHÉS},
    doi = {10.1007/bf02684599},
    publisher = {Springer Science and Business Media LLC},
    month = {January},
}

@misc{viviani,
    archiveprefix = {arXiv},
    author = {Filippo Viviani},
    eprint = {1204.3875},
    primaryclass = {math.AG},
    year = {2013},
    title = {Tropicalizing vs Compactifying the Torelli morphism},
}

@misc{linsys-independence,
    author = {David Jensen and Sam Payne},
    eprint = {2209.15478},
    archiveprefix = {arXiv},
    year = {2022},
    primaryclass = {math.AG},
    title = {Tropical Linear Series and Tropical Independence},
}

@misc{luo-idempotent,
    title = {Idempotent Analysis, Tropical Convexity and Reduced Divisors},
    author = {Ye Luo},
    primaryclass = {math.CO},
    archiveprefix = {arXiv},
    year = {2018},
    eprint = {1808.01987},
}

@misc{kodaira-dimensions,
    title = {The Kodaira dimensions of $\overline{\mathcal{M}}_{22}$ and $\overline{\mathcal{M}}_{23}$},
    author = {Gavril Farkas and David Jensen and Sam Payne},
    year = {2023},
    primaryclass = {math.AG},
    archiveprefix = {arXiv},
    eprint = {2005.00622},
}

@article{capacity-pairing,
    author = {Ted CHINBURG and Robert Rumely},
    lastchecked = {2024-06-19},
    volume = {1993},
    number = {434},
    pages = {1--44},
    doi = {doi:10.1515/crll.1993.434.1},
    year = {1993},
    url = {https://doi.org/10.1515/crll.1993.434.1},
    title = {},
    journal = {Journal für die reine und angewandte Mathematik},
}

@book{hartshorne,
    isbn = {9781475738490},
    doi = {10.1007/978-1-4757-3849-0},
    issn = {2197-5612},
    url = {http://dx.doi.org/10.1007/978-1-4757-3849-0},
    author = {Hartshorne, Robin},
    publisher = {Springer New York},
    title = {Algebraic Geometry},
    year = {1977},
    journal = {Graduate Texts in Mathematics},
}

\end{document}